\documentclass[aos,preprint]{imsart}
\RequirePackage[OT1]{fontenc}
\RequirePackage{amsthm,amsmath,amsfonts,amssymb}
\RequirePackage[numbers]{natbib}

\RequirePackage[colorlinks,citecolor=blue,urlcolor=blue]{hyperref}
\RequirePackage{graphicx}

\startlocaldefs

\newtheorem{theorem}{Theorem}[section]
\newtheorem{lemma}[theorem]{Lemma}
\theoremstyle{remark}

\endlocaldefs

\RequirePackage{myMATH}
\RequirePackage{graphicx, morefloats}
\RequirePackage{lineno}
\newtheorem{assumption}{Assumption}
\newtheorem{proposition}{Proposition}

\newenvironment{supplemma}[1]
  {\renewcommand\thelemma{#1}\lemma}
  {\endlemma}

\usepackage[usenames,dvipsnames]{xcolor}
\newcommand{\h}[1]{\textcolor{black}{#1}}

\def\segp{\textrm{SEGP}}
\def\rateOne{\underline{\varepsilon}}
\def\rateTwo{\varepsilon}
\def\fn{\mathrm{FN}}
\def\fp{\mathrm{FP}}

\graphicspath{{./figures/}}

\begin{document}

\begin{frontmatter}
  \title{Variable Selection Consistency of Gaussian Process Regression } 
  \runtitle{GP VS Consistency}

  \begin{aug}
    \author{\fnms{Sheng} \snm{Jiang}\ead[label=e1]{sheng.jiang@duke.edu}}
    \and 
    \author{\fnms{Surya T.} \snm{Tokdar}\ead[label=e2]{surya.tokdar@duke.edu}}



    \address{Department of Statistical Science, 
      Duke University,
      \printead{e1,e2}}
    
  \end{aug}


  \maketitle   
  \begin{abstract}
    Bayesian nonparametric regression under a rescaled Gaussian
    process prior offers smoothness-adaptive function estimation with
    near minimax-optimal error rates. Hierarchical extensions of this
    approach, equipped with stochastic variable selection, are known
    to also adapt to the unknown intrinsic dimension of a sparse true
    regression function. But it remains unclear if such extensions
    offer variable selection consistency, i.e., if the true subset of
    important variables could be consistently learned from the
    data. It is shown here that variable consistency may indeed be
    achieved with such models at least when the true regression
    function has finite smoothness to induce a polynomially larger penalty
    on inclusion of false positive predictors.
    Our result covers the high dimensional asymptotic setting
    where the predictor dimension is allowed to grow with the sample size.
    The proof utilizes Schwartz theory to establish that
    the posterior probability of wrong selection vanishes asymptotically.
    A necessary and challenging technical development involves
    providing sharp upper and lower bounds to small ball probabilities
    at all rescaling levels of the Gaussian process prior,
    a result that could be of independent interest. 
  \end{abstract}



  
\end{frontmatter}

\section{Introduction} 
\label{sec:Intro} 
\h{Sparse estimation and variable selection in high dimensional linear regression has been well studied 
\citep{wainwright2009information, buhlmann2011statistics, verzelen2012minimax, narisetty2014bayesian,
castillo2015bayesian, hahn2015decoupling, 
  yang2016computational, wainwright2019high}.
But an
assumption of linearity could be overly restrictive and prone to model misspecification.} A natural alternative is to allow the predictor-response relationship to be flexibly determined by an unknown smooth function $f:\R^d \to \R$, leading to the nonparametric regression model
\begin{equation}
  \label{eq:NPreg}
  Y_i=f(X_i) + \epsilon_i, \quad\quad \epsilon_i \mid X_i \distras{iid} N(0,\sigma^2),
\end{equation}
for paired data $(X_i, Y_i) \in \R^d \times \R$, $i = 1, \ldots, n$.
No theoretical results currently exist on simultaneously estimating $f$ and recovering its sparsity pattern, especially in the high dimensional setting. Earlier results restrict to the low dimensional settings of $d = o(\log n)$ or $d = O(\log n)$ \citep{zhang1991variable, lafferty2008rodeo, bertin2008selection}. \citet{comminges2012tight} present results under more relaxed settings, establishing that variable selection consistency is achievable for designs of size $\log(d) = o(n)$. However, they focus exclusively on understanding when it is possible to consistently recover the sparsity pattern of $f$, rather than providing a practicable estimation method of either the function or the sparsity pattern. More recent works have mostly concentrated on sparse additive formulations of $f$ \citep{raskutti2012minimax,butucea2017adaptive}.

In order to estimate $f$ from data, it is attractive to adopt a Bayesian approach where a Gaussian process prior is specified on the unknown regression function. Fairly expressive yet computationally tractable estimation models could be developed by specifying hierarchical extensions of this prior with rescaling and stochastic variable selection in order to infer smoothness and sparsity properties of $f$, and consequently, improve function estimation and prediction \citep{rasmussen2006gaussian, ghosal2017fundamentals}. Indeed, under such Bayesian approaches, the posterior distribution of the regression function is known to contract to the true function at a near minimax optimal rate,
adapting to both unknown smoothness and sparsity \citep{van2009adaptive,tokdar2011dimension, 
  bhattacharya2014anisotropic, yang2015minimax}. 

\h{Does this adaptive regression function estimation accuracy translate to correct identification of relevant regressors?} 
We show that the answer is partially {\it yes}. 
Specifically, we establish that when the true sparsity pattern has a fixed cardinality $d_0$ and the true regression function is Sobolev $\beta$-smooth, an appropriately specified rescaled Gaussian process model with stochastic variable selection simultaneously estimates the regression function with near minimax optimal estimation accuracy and recovers the correct sparsity pattern with probability converging to one, provided $\log (d) \le O(n^{d_0/(2\beta + d_0)})$ and the predictors are independent Gaussian variables with a common variance.

Like \cite{comminges2012tight}, our result covers the case where the predictor dimension may grow much faster than the sample size, but ties the growth rate with the smoothness level of the true regression function. Larger design dimensions  are allowed when the true function is more rough. Both this smoothness dependent bound on the predictor dimension and the independent Gaussian design assumption are necessary for our proof technique which relies on establishing a polynomial slowdown in the posterior contraction rate of a rescaled Gaussian process model when it includes more predictors than necessary.

Establishing this polynomial slowdown in contraction rate requires new and detailed calculations of concentration probabilities of a rescaled Gaussian process around a given function of limited {Sobolev} smoothness. 
To prove that a rescaled Gaussian process estimation model achieves near minimax-optimal contraction rate, \citet{van2009adaptive} derived lower bounds on the concentration probabilities for a carefully chosen range of rescaling levels. To establish a polynomial slowdown in rate, we additionally need sharp \textit{upper} bounds on the concentration probabilities at \textit{every} rescaling level.  
These 
are derived by an accurate characterization of {\it small $L_2$-ball probabilities} of a rescaled Gaussian process at \textit{every} rescaling level. 
We use the metric entropy method
\citep{kuelbs1993metric,li1999approximation, van2008reproducing} 
to turn small ball probability calculations into calculating the metric entropy of the unit ball of the Reproducing Kernel Hilbert Space (RKHS) associated with the Gaussian process. Critically, the Gaussian design assumption allows mapping the RKHS of a squared-exponential Gaussian process 
to an $\ell_2$ ellipsoid whose metric entropy can be bounded accurately.

\relax

\section{GP Regression with Stochastic Variable Selection} 
\label{sec:model-setup} 

Our asymptotic analysis concerns a sequence of experiments, indexed by sample size $n = 1, 2, \ldots$, 
and with associated sample spaces
$\mathcal{D}_n = \{(X_i, Y_i) \in \R^{d_n} \times \R, 1 \le i \le n\}$, 
in which $X_i$'s are taken to be independent realizations 
from a probability measure $Q_n$ on $\R^{d_n}$ 
and $Y_i$'s are realized as
in the nonparametric regression formulation \eqref{eq:NPreg}
for some known $\sigma > 0$ and some unknown $f \in C(\R^{d_n})$,
the space of continuous functions on $\R^{d_n}$. We study a Bayesian estimation of $f$ under a rescaled Gaussian process prior with stochastic variable selection. This prior is formalized by hierarchically extending the rescaled Gaussian process prior of \citep{van2009adaptive}; \h{see also \cite{yang2015minimax}}.

\subsection{Prior specification}
\label{sec:GPprior}
We call a stochastic process $W = (W(x): x \in \R^d)$ a standard, squared exponential Gaussian process on $\R^d$, if for any finite collection $\{x_1, \ldots, x_k\} \subset \R^d$, the random vector $(W(x_1), \ldots, W(x_k))$ has a mean-zero, $k$-variate Gaussian distribution with covariance matrix $((K(x_i, x_j))_{1 \le i, j \le k}$ where $K(x, x') = \exp\{-\|x - x'\|_2^2\}$. It is well known that such a $W$ could be seen as a random element of $C(\R^d)$ and is infinitely differentiable with probability one.

Now, given a $\gamma \in \{0,1\}^d$ and a probability measure $\pi$ on $(0,\infty)$, 
define a $\gamma$-sparse, $\pi$-rescaled, squared exponential Gaussian process measure, \segp$(\R^d; \gamma, \pi)$ for short, as the probability law of the stochastic process $W = (W(x): x\in \R^d)$ given as $W(x) = W_0(A x_\gamma)$ where $W_0$ is a standard, squared exponential Gaussian process on $\R^{|\gamma|}$, $A \sim \pi$ independently of $W_0$,
and $x_\gamma = (x_j : \gamma_j = 1,  j = 1, \ldots, d)$
denotes the sub-vector selected according to $\gamma$.
The rescaling measure $\pi(A)$ plays a key role in smoothness adaptation, facilitating a Bayesian version of fully automated, data-driven bandwidth selection \citep{van2009adaptive}.

For our experiment sequence $\cD_n$, a rescaled Gaussian process prior on $f \in C(\R^{d_n})$  with stochastic variable selection is defined as the marginal law of $f$ induced by any joint probability measure $\Pi_n$ on $(\Gamma, f) \in \{0,1\}^{d_n} \times C(\R^{d_n})$ satisfying the following
\begin{enumerate}
\item $\Pr(\Gamma = \gamma) = q_n(|\gamma|) / {d_n \choose |\gamma|}$, $\gamma \in \{0,1\}^{d_n}$, for some probability vector $(q_n(d): 0 \le d \le d_n)$ with $q_n(d_n) < 1$.
\item For every $\gamma \in \{0,1\}^{d_n}$,
  $f \mid (\Gamma = \gamma) \sim \segp(\R^{d_n}; \gamma, \pi_{n,|\gamma|})$,
  determined by a collection of probability measures $\pi_{n,d}$ on $(0,\infty)$, $0 \le d \le d_n$.
\end{enumerate}

The sparsity pattern of $f$ is fully encoded by the binary vector $\Gamma$. Let $Q_{n,j}$ denote the marginal distribution of the $j$-th regressor under $X \sim Q_n$. Any $f \in C(\R^{d_n})$ is constant along an axis $j \in \{1, \ldots, d_n\}$ if and only if $\|f - f_j\|_{L_2(Q_n)} = 0$
where $f_j(x) := \int f(x_1, \ldots, x_{j-1}, z, x_{j+1}, \ldots, x_{d_n}) dQ_{n,j}(z)$.
Under the prior $\Pi_n$, the sparsity pattern of $f$ given by the subset $\{1 \le j \le d_n : f~\mbox{is not constant along axis}~j\}$ is identical to $\Gamma$ with probability one. Notice that the prior on variable selection is taken to depend only on the cardinality of the included subset.

\subsection{Connecting selection consistency with  estimation accuracy}
\label{sec:general-approach}

Given observed data $D_n \in \mathcal{D}_n$, let $\Pi_n(\cdot \mid D_n)$ denote the joint posterior distribution on $(\Gamma, f)$ under a prior distribution $\Pi_n$ as above. Assuming $D_n$ was generated from a {\it true} regression function $f^*_n $  
whose sparsity pattern is identified by a $\gamma^*_n \in \{0,1\}^{d_n}$, the issue of variable selection consistency boils down to assessing whether $\Pi_n(\Gamma \neq \gamma^*_n \mid D_n) \to 0$ in some probabilistic manner.  
As indicated in the Introduction, the main result we prove in this paper is that when the cardinality of $\gamma^*_n$ remains fixed at a $d_0$ and $f^*_n$ is of finite Sobolev smoothness of order $\beta$, one has $\Pi_n(\Gamma \ne \gamma^*_n \mid D_n) \to 0$
as long as $\log (d_n) \precsim n^{d_0/(2\beta + d_0)}$ and the regressors are independent Gaussian variables with equal variance. Below we give a sketch of the argument of how such a claim can be made based on adaptive function estimation accuracy.


The question of variable selection 
could be directly related to that of function estimation quality as follows. Consider a small ball around the truth: $E_n = \{f : \rho(f, f^*_n) \le \rateTwo_n\}$ where $\rho$ is an appropriate metric and $\rateTwo_n > 0$. Notice that
\begin{equation}
  \label{eq:main-inequality}
  \Pi_n (\Gamma  \ne \gamma^*_n \mid D_n ) \le \Pi_n ( f \in E_n ^c  \mid D_n) +  \Pi_n(\Gamma  \ne \gamma^*_n, f \in E_n \mid D_n).
\end{equation} 
If it were known that the posterior on $f$ contracts to the truth under metric $\rho$ at a rate $\rateTwo_n$ or faster, then the first term on the right hand side would eventually vanish in probability. In order for the second term to vanish as well, one needs to establish that the same fast rate of contraction cannot be achieved under a wrong selection of variables. 

Partition the space of wrong selections into two parts:
$\{0,1\}^{d_n} \setminus \{\gamma^*_n \}
= \{\gamma \in \{0,1\}^{d_n}:
\gamma^*_n \not\leq \gamma\}
\cup \{\gamma \in \{0,1\}^{d_n}: \gamma^*_n < \gamma\}
=: \fn(\gamma^*_n) \cup \fp(\gamma^*_n)$, 
where the defining inequalities are taken to be coordinate-wise.
The {\it false negative} set $\fn(\gamma^*)$ 
consists of selections that miss at least one true predictor.
The {\it false positive} set $\fp(\gamma^*)$ contains selections that include all important variables and at least one unimportant regressor. 
Accordingly, the second posterior probability in \eqref{eq:main-inequality} splits into two pieces,
\begin{equation}
\begin{split}
 \Pi_n(\Gamma  \ne \gamma^*_n, f \in E_n \mid D_n) = \Pi_n&(\Gamma  \in \fn(\gamma^*_n), f \in E_n \mid D_n)\\
  & + \Pi_n(\Gamma  \in \fp(\gamma^*_n), f \in E_n \mid D_n),
\end{split} 
  \label{eq:2pieces}
\end{equation}
of which the first term could be expected to be exactly zero for large $n$ as long as one assumes that the signal strength of $f^*_n$ in any of its relevant variables, measured according to $\rho$, is above a fixed threshold $\delta_n \equiv \delta$ (Assumption \ref{assumption:sparsity} in Section \ref{sec:formal-statements}). Such an $f^*_n$ will be at least a $\delta$ distance away in metric $\rho$ from any $f$ whose sparsity pattern $\gamma \in \fn(\gamma^*_n)$.

No such separation exists for regression functions associated with selections in the false positive set.
For any $\gamma \in \fp(\gamma^*_n)$, one would expect the conditional posterior $\Pi_n(f \mid \Gamma = \gamma, D_n)$ to place considerable mass around the truth $f^*_n$. Any hope of the second term on the right hand side of \eqref{eq:2pieces} being small rests on establishing that such a conditional posterior would contract at a slower rate than $\rateTwo_n$. This is a legitimate hope because it is known that the minimax error rate of function estimation usually worsens when additional irrelevant variables are selected in the regression model \citep{stone1982optimal}. Therefore, assuming $f^*_n$ belongs to the class of $\beta$-smooth functions for some $\beta > 0$ (Assumption \ref{assumption:smoothness} part 1), a reasonable proof strategy would be to consider $\rateTwo_n = M \rateOne_n (\log n)^\kappa$ for some $M, \kappa > 0$ where $\rateOne_n = n^{-\beta / (2\beta + d_0)}$ is the usual minimax rate of estimation of an $f^*$ of $\beta$-smooth functions, and establish a polynomial difference in contraction rates between the overall posterior and the conditional posteriors under false positive selection.

In Sections \ref{sec:formal-statements} and \ref{sec:main-result} we establish the above rate difference result and give a formal proof of variable selection consistency under a rescaled GP prior with stochastic variable selection. Establishing slower posterior contraction rates under false positive selection necessitates working with the $L_2(Q_n)$ metric $\rho(f, f') = \|f - f'\|_{L_2(Q_n)} = \{\int (f - f')^2 dQ_n\}^{1/2}$. This choice of metric is slightly different than those considered in \citep{van2009adaptive, yang2015minimax}. Consequently, a new proof is needed to establish overall posterior contraction at a rate of $\rateTwo_n$. To show that posterior contraction rate becomes polynomially slower under false positive selection, we have to make two important assumptions: $f^*_n$ is exactly $\beta$-smooth and no smoother (Assumption \ref{assumption:smoothness} part 2) and $Q_n$ is the mean-zero Gaussian measure on $\R^{d_n}$ with covariance matrix $\xi^2\cdot I_{d_n}$ (Assumption \ref{assumption:q}). The finite smoothness of $f^*_n$ is indeed necessary to ensure that the conditional prior $\Pi(f \mid \Gamma = \gamma)$ sits less densely around $f^*_n$ when $\gamma \in \fp(\gamma^*_n)$ than when $\gamma = \gamma^*_n$. But the assumption of an independent Gaussian design with a fixed variance is a technical convenience that enables sharp calculations of the concentration probabilities of rescaled GP laws, which are necessary for establishing the above result. Additionally, the stochastic variable selection prior is assumed to favor small models, to control the total posterior probability of the exponentially growing false positive set (Assumption \ref{assumption:Gamma}). 

\relax

\section{Main Result} 
\label{sec:formal-statements}
Toward a formal and rigorous treatment of the arguments presented above, we first state the necessary assumptions and the main variable consistency result.  \h{A lengthy discussion of the assumptions is delayed to Section \ref{sec:discussion}. Supporting results on the posterior contraction rates and difference in such rates under correct and false positive selections are presented in Section \ref{sec:main-result}, relying upon the sharp small ball probability calculations in Section \ref{sec:sbp}.}


An important assumption, needed essentially for technical reasons, is that the design distribution is uncorrelated Gaussian.

\begin{assumption}[Gaussian random design]
\label{assumption:q}
The design measure $Q_n = G_{d_n}$ where $G_d$ denotes the $d$-variate Gaussian measure
$N_d(0, \xi^2 I_{d})$. 
The variance $\xi^2 >2/e$ may be unknown but does not change with $n$.
\end{assumption}

For each sample size $n$, the true data generating distribution $\bP^*_n$ is taken to be an element of the model space identified by $f = f^*_{n} \in \cL_2(Q_n)$ where  
$\cL_2(Q_n) := C(\R^{d_n}) \cap L_2(Q_n)$.  

Additionally, we assume that the sequence $(f^*_{n}: n \ge 1)$ remains essentially the same across $n$, formalized as below. 
For notational convenience, for any $d \in \N$ and $\gamma \in \{0,1\}^d$, let $T_\gamma: C(\R^{|\gamma|}) \to C(\R^d)$ denote the function embedding operator: $(T_\gamma f)(x) = f(x_\gamma)$, $f \in C(\R^{|\gamma|})$, $x \in \R^d$.
With $T_\gamma$, any high dimensional functions with redundant variables can be decomposed into a low dimensional function without any redundant variables and a variable inclusion vector. 
Smoothness conditions are directly imposed to the low dimensional functions;
sparsity and dimensionality assumptions are made on the variable inclusion vector. 

\begin{assumption}[Finite sparsity of true regression function]
\label{assumption:f0}
There exist $n_0, d_0 \in \N$, $f_0 \in \cL_2(Q_0)$ and a sequence of binary vectors $\gamma^*_{n} \in \{0,1\}^{d_n}$, such that $d_n \ge d_0$, $|\gamma^*_n| = d_0$ and $f^*_n = T_{\gamma^*_n} f_0$ for all $n \ge n_0$.
\end{assumption}

Assumption \ref{assumption:f0} makes it clear that for all large $n$, the true function is sparse and the support size $d_0$ does not grow with sample size. To avoid any ambiguity about the true sparsity level $d_0$, it is important to identify it as the minimal support size for the sequence $(f^*_n: n \ge 1)$. This is done via the next assumption on signal strength which ensures that each of the $d_0$ inputs to $f_0$ results in a variability that is detectable in the $L_2$ topology. Toward this, for each $j \in \{1, \ldots, d_0\}$, define $f_{0j} \in \cL_2(Q_0)$ as the projection of $f_0$ perpendicular to the $j$-th axis, given by $f_{0j}(x) := \int_\R f_0(x_1, \ldots, x_{j-1}, z, x_{j+1}, \ldots, x_{d_0}) dG_1(z)$, $x \in \R^{d_0}$.

\begin{assumption}[Signal strength is $L_2$ detectable]
  \label{assumption:sparsity} 
The minimum signal strength in the relevant variables $\delta := \min_{1\le j \le d_0} \|f_0 - f_{0j}\|_{L_2(G_{d_0})}^2$ is strictly positive.
\end{assumption}  
An immediate consequence of  Assumption \ref{assumption:sparsity} is that for any $n \ge n_0$,  
$f^*_n$ is at least a $\delta$ distance away from any $f \in \cL_2(Q_n)$ 
that is constant along at least one axis $j$ for which $\gamma^*_{n,j} = 1$. 
More formally, for any $n \ge n_0$, and any $\gamma \in \{0,1\}^{d_n}$ with $\gamma^*_n \not\le \gamma$, 
$\inf \{\|f^*_n - f\|_{L_2(Q_n)}^2: f \in T_\gamma C(\R^{|\gamma|}) 
\cap \cL_2(Q_n) \}\ge \delta$. 
To see why, without loss of generality, consider the toy example of $f_0(x_1,x_2)$ and $f(x_1,x_3)$ where $x_1$ and $x_2$ are
  correct variables and $x_3$ is the redundant variable.
  Then we can compute 
  $\|f_0 - f\|_{L_2(Q)}^2 = \bE[(f_0(X_1,X_2) - f(X_1,X_3))^2]
  =\bE[(f_0(X_1,X_2) -f_{0,2}(X_1) + f_{0,2}(X_1) -  f(X_1,X_3))^2]  
  $
  whose cross-term 
  $\bE[(f_0(X_1,X_2) -f_{0,2}(X_1))
  (f_{0,2}(X_1) -  f(X_1,X_3))] = 0$ holds because
  $X_i$'s are independent.

Next, we formalize the notion that the true regression function is $\beta$-smooth but no smoother.  For any $d \in \N$, $\beta > 0$, let $H^\beta(\R^d)$ denote the Sobolev space of functions $h : \R^d \to \R$ with norm $\|h\|_{H^\beta (\R^d)}$ given by
  \begin{equation}
    \label{eq:sobolev-norm}
    \|h\|^2_{H^\beta (\R^d)} := \int_{\R^d}
    |\hat h (\lambda)|^2 (1 + \|\lambda\|_2^2)^\beta d\lambda < \infty    
  \end{equation}
where $\hat h$
is the Fourier transform\footnote{$\hat h(\lambda) = (2\pi)^{-d} \int_{\R^d} e^{-i(\lambda, t)} h(t)dt$, $\lambda \in \R^d$.} of $h$. Recall that functions $h \in H^\beta(\R^d)$ have square-integrable, (weak) derivatives $D^{(k)}f$, $k = (k_1, \ldots, k_d) \in \N_0^d$, of order $|k| \le \beta$.  
\begin{assumption}[Smoothness of $f_0$]
  \label{assumption:smoothness}
  The true function $f_0$ satisfies 
  \begin{enumerate}
  \item There exists a $\beta > d_0/2$ such that $f_0 \in H^\beta(\R^{d_0}) \cap L_2(Q_0)$. 
  \item There also exists an $\alpha \in (\beta, \beta(1 + 1/d_0))$ such that  $|\widehat {f_0 \sqrt{ g_{d_0}}} (\lambda)|  
    \succsim \|\lambda\|^{-(\alpha+d_0/2)}_2 $ 
    for every $\lambda \in \R^{d_0}$ with $\|\lambda\|_2 \ge 1$,
    where $g_{d_0}$ is the probability density function of $G_{d_0}$. 
  \end{enumerate}
\end{assumption}
Part 2 of the assumption ensures that $f_0 \not\in H^b(\R^{d_0})$ for any $b > \alpha$ and hence has limited regularity which is important in establishing that the posterior contraction rate at $f^*_n$ is polynomially slower under false positive inclusion. 

The posterior contraction rates also depend on the rescaling measures $\pi_{n,d}$, $0 \le d \le d_n$. The following assumption is mildly adapted from \citep{van2009adaptive}. The modification is needed in part because in determining sharp upper bounds on the concentration probabilities of rescaled GP priors, one needs to integrate over the entire range of the rescaling parameter. Below, with a slight abuse of notation, we let $\pi_{n,d}$ also denote the probability density function underlying the eponymous rescaling measure. 
\begin{assumption}[Rescaling measures] 
  \label{assumption:A} For each $d \in \N$, there exist constants
  $C_1$, $C_2$, $C_3$, $D_1$ and $D_2$, all independent of $d$, such that
  \begin{enumerate}
  \item for all sufficiently large $a$, 
    $\pi_{n,d} (a)  \ge D_1 e^{ - C_1 a^d \log^{d+1} (a)}$; 
  \item for every $a>1/\xi$,
    $\pi_{n,d}(a) \le D_2 a^{d-1} e^{ - C_2 a^d (\log^{d+1} (a) \vee 1) + C_3 \log (d)}$;
  \item $\pi_{n,d}(0, \xi^{-1}) =0$. 
\end{enumerate}  
\end{assumption} 
This assumption is satisfied, for example, when $\pi_{n,d}$ is the truncation to $(\xi^{-1}, \infty)$ of the probability law of a random variable $A$ for which $A^{d}\log^{d+1}(A)$ has an exponential distribution with a rate parameter that is constant in $n$ and $d$.

Our next assumption regulates how fast the design dimension $d_n$ can grow in $n$. If the support $\gamma^*_n$ were known, the $L_2(Q_n)$ minimax estimation error of a $\beta$-smooth function would be of the order of $\rateOne_n = n^{ - 1/( {2 + {d_0}/\beta })}$ \citep{stone1982optimal}. Not knowing the support means that one incurs an additional error of the order of $d_0 \log(d_n/d_0) / n$ for having to carry out variable selection. We require this additional error to not overwhelm the original estimation error.
\begin{assumption}[Growth of $d_n$] 
  \label{assumption:d}  
 The design dimension $d_n$ satisfies 
 $\log (d_n) \precsim n \rateOne_n^2 \asymp n^{d_0/(2\beta + d_0)}$. 
\end{assumption} 

A final assumption is needed on the sparsity induced by the prior distribution. In particular it is needed that the prior on $\Gamma$ favors small selection sizes and heavily penalizes extremely large selections. 
\begin{assumption}[Prior sparsity] 
  \label{assumption:Gamma} 
  For all sufficiently large $n$,
  \begin{enumerate}
  \item  $q_n(d_0) \ge \exp\{-n\rateOne_n^2\} \asymp \exp\{-n^{d_0 / (2\beta + d_0)}\}$,
  \item $q_n(d) \le \exp\{-C d^\rho\}$ for every $d \succsim  n^{2\beta/\{\alpha(2\beta+d_0)\}}$, for some constants $C > 0$ and $\rho \ge (d_0 + 1)/2$. 
  \end{enumerate}
\end{assumption} 
Assumption \ref{assumption:Gamma} seemingly requires the knowledge of the true support size $d_0$,
but one can relax this by letting $\rho$ grow slowly as sample size increases. 
The assumption would hold, for instance if one chose a prior that caps the selection size $|\Gamma|$ at an $m_n \le d_n$ and let $m_n$ grow slowly with $n$, e.g., $m_n \asymp n^{1/\log\log n}$.
Formally, $q_n(d) \propto I(d < n^{1/\log\log n})$, $d = 0, \ldots, d_n$. An alternative is to not use a cap, but employ aggressive penalization of larger selections: $q_n(d) \propto d^{k \log\log n -1}\exp\{-d^{k\log\log n}\}$, $0 \le d \le d_n$, for some constant $k$. The latter choice is equivalent to an appropriately tuned Beta-Binomial prior on individual regressor inclusion. 

Building upon these formal assumptions, we are able to offer the following rigorous statement and proof of variable selection consistency.
\begin{theorem}
  \label{thm:main}
Under Assumptions \ref{assumption:q}-\ref{assumption:Gamma},
$\bP^*_n \left[ \Pi_n (\Gamma  \ne \gamma^*_{n} \mid D_n) \right] \to 0$,
as $n \to \infty$.
\end{theorem}  

\begin{proof}
  As before, let $E_n = \{f \in \cL_2(Q_n): \|f - f^*_n\|_{L_2(Q_n)} \le \rateTwo_n\}$ where $\rateTwo_n = \rateOne_n (\log n)^\kappa$ with $\kappa = (d_0 + 1)/(2 + d_0/\beta)$. Consider the upper bound on $\Pi_n(\Gamma \neq \gamma^*_n \mid D_n)$ jointly given in \eqref{eq:main-inequality}
  and \eqref{eq:2pieces}. By Theorem \ref{thm:overall-contraction} in the next section, $\Pi_n(f \in E_n^c \mid D_n) \to 0$ in probability as $n \to \infty$. By Assumption \ref{assumption:sparsity}, the first term of the bound given by \eqref{eq:2pieces} is exactly zero for all large $n$ because the prior probability $\Pi_n(\Gamma \in \fn(\gamma^*_n), f \in E_n) = 0$ whenever $\rateTwo_n < \delta$. The second piece of this bound vanishes in probability by Proposition \ref{proposition:FP-prior-prob}, which leverages on detailed calculations of concentration properties of Gaussian process laws presented in Section \ref{sec:sbp}.
\end{proof}
\relax

\section{Posterior Concentration via Schwartz Theory}  
\label{sec:main-result}
This Section presents supporting results for the proof of Theorem \ref{thm:main}.
As the proof technique is standard, details of the proofs are in the Supplement. 

\begin{theorem}
  \label{thm:overall-contraction}
  Under Assumptions \ref{assumption:q}, \ref{assumption:f0}, \ref{assumption:smoothness}-1, \ref{assumption:d} and \ref{assumption:A},
  let $\rateTwo_n = \rateOne_n (\log n)^\kappa$ 
  with $\kappa = (d_0+1)/\left( {2 + {d_0}/\beta } \right)$ 
  for $n \ge 1$, 
  then for any sufficiently large constant $M$,
  \[ \bP_n^* [\Pi_n ( f \in \cL_2(Q_n):
    \|f - f^*_n\|_{L_2(Q_n)} > M \rateTwo_n\}
    \mid D_n)] 
    \to 0, 
    \text{ as } n\to \infty. \]
\end{theorem}
A proof of this result is presented in the Supplement 
by verifying Theorem 2.1 of \cite{ghosal2000convergence}.
In the proof, we first verify the Kullback-Leibler prior mass condition that for all sufficiently large $n$,
\begin{equation}
\label{eq:kl}
\Pi_n(f \in B_n(f_n^*,\rateTwo_n)) \ge e^{- n\rateTwo _n^2}
\end{equation}
where for any $g \in \cL_2(Q_n)$ and $\epsilon > 0$,
one defines $B_n(g, \epsilon) = \{f : K(\bP^1_g, \bP^1_f) \le \epsilon^2,
V(\bP^1_g, \bP^1_f) \le \epsilon^2\}$,
with $\bP^1_f$ denoting the probability distribution of a single observation pair $(X_1, Y_1)$ under the model element with regression function $f$,
and, $K(\bP^1_g, \bP^1_f) = \bP^1_g \log (d\bP^1_g/d\bP^1_f)$,
$V(\bP^1_g, \bP^1_f) = \bP^1_g [(\log (d\bP^1_g/d\bP^1_f))^2]
- K(\bP^1_g, \bP^1_f)^2$.
Next, a sieve $\bB_n \subset L_2(Q_n)$, $n \in \N$ is produced such that 
$\Pi_n ( f \notin \bB_n) \le \exp( - 4n\rateTwo_n^2)$
and $\log N( \rateTwo _n, \bB_n ,\| \cdot \|_{L_2(Q_n)}) \le n\rateTwo_n^2$
for all sufficiently large $n$.
Here, $N(\epsilon, S, \rho)$ is used to denote the $\epsilon$-covering number of a subset $S$ in a metric space with metric $\rho$.

Our statement and proof technique for Theorem \ref{thm:overall-contraction} mirror the near minimax-optimal posterior contraction results and proofs of \cite{van2009adaptive, yang2015minimax}, but are slightly novel in its use of the $L_2(Q_n)$ topology on the space of regression functions. In particular, under the stochastic design assumption, one has the simplification that $B_n(g, \epsilon) = \{f \in \cL_2(Q_n): \|f - g\|_{L_2(Q_n)} \le \epsilon\})$. Thus the Kullback-Leibler prior mass condition translates to a simpler $L_2(Q_n)$ prior mass condition:
\begin{equation}
\label{eq:l2}
\Pi_n(\{f \in \cL_2(Q_n) : \|f - f^*_n\|_{L_2(Q_n)} \le \rateTwo_n\}) \ge e^{-n\epsilon_n^2}.
\end{equation}

\begin{proposition}
  \label{proposition:FP-prior-prob}
    Under Assumptions \ref{assumption:q}, 
  \ref{assumption:f0},
  \ref{assumption:smoothness}-2, 
  \ref{assumption:A} and \ref{assumption:Gamma},
  let $E_n = \{f \in \cL_2(Q_n): \|f - f^*_n\|_{L_2(Q_n)} \le M \rateTwo_n\}$, 
   $\rateTwo_n = \rateOne_n (\log n)^\kappa$ 
  with $\kappa = (d_0+1)/\left( {2 + {d_0}/\beta } \right)$ 
  for $n \ge 1$, 
  then 
  for all sufficiently large constant $M$, 
  \[ \bP^*_n [ \Pi_n(\Gamma  \in \fp(\gamma^*_n),
    f \in E_n \mid D_n)] 
    \to 0,
    \text{ as } n\to \infty. \] 
\end{proposition}
A proof is given in the Supplement. 
It follows Lemma 1 of \cite{castillo2008lower} and establishes that   
\begin{equation}
  \label{eq:castillo2008} 
\frac{  \Pi_n (\Gamma \in \fp (\gamma^*_n), f \in  E_n)}
  {\Pi_n (f \in B_n(f_n^*,\rateTwo_n) )}
  \le 
  {e^{ - 2 n\rateTwo _n^2}},
\end{equation}
which, according to Lemma 1 in \cite{ghosal2007convergence}, along
with \eqref{eq:kl}, guarantees that the posterior probability of the
set in the numerator vanishes in probability.

The message of Proposition \ref{proposition:FP-prior-prob} is that
the posterior contraction rates of false positive models are
significantly slower than minimax rates.
With Theorem \ref{thm:overall-contraction}, it further
implies false positive models eventually receive negligible
posterior mass. 
The use of posterior contraction rates adaptation
echoes \cite{ghosal2008nonparametric}
who establishes model selection consistency among density models with
different regularity levels.

The key inequalities in \eqref{eq:l2} and \eqref{eq:castillo2008} require establishing lower and upper bounds on the prior concentration in $L_2(Q_n)$ balls around $f^*_n$. Such concentrations are completely governed by the {\it $f^*_n$-shifted small ball probabilities} of the underlying sparse Gaussian processes at appropriate rescaling levels. The lower bound calculations are similar to those in \cite{van2009adaptive} and require working with rescaling levels that appropriately grow to infinity as determined by the true smoothness level $\beta$. But, as mentioned earlier, the upper bound calculations are much more technically involved and require dealing with all rescaling levels. Furthermore, unlike \cite{van2009adaptive} we calculate small ball probabilities by viewing the Gaussian process as a random element in $L_2(Q_n)$, necessitating new characterization of the associated reproducing kernel Hilbert space. These details are presented in the next section.

\relax

\section{Small Ball Probability of Rescaled GP}   
\label{sec:sbp}

\subsection{Series representation of $W^{a,\gamma}$}
\label{sec:KL-expansion}
Let ${W^{a,\gamma}} \sim \segp (\R^{d_n};\gamma, \delta_ a)$ where
$\gamma \in \{0,1\}^{d_n}$, 
$|\gamma|>0$, and $a>0$ is
a fixed rescaling level, i.e., the rescaling measure is the Dirac delta measure at level $a$. 
In this section, we obtain a series representation of $W^{a,\gamma}$ via the 
Karhunen-Lo$\grave{e}$ve expansion of the covariance kernel of $W^{a,\gamma}$.

The covariance kernel of $W^{a,\gamma}$ is
\begin{equation}
  \label{eq:SE-kernel}
    {K_{a,\gamma} }\left( {s,t} \right)
    = {e^{ - {a^2}\|s_{[\gamma]} - t_{[\gamma]}\|_2^2}}
    = \prod\nolimits_{\{i: \gamma_i=1\} }^{}
        {{e^{ - {a^2}{{\left( {{s_i} - {t_i}} \right)}^2}}}}.  
\end{equation} 
To obtain an eigen-expansion of $K_{a,\gamma}(\cdot,\cdot)$, 
first recall the eigen-expansion of
univariate Squared Exponential kernel function under Gaussian design. 
For $s,t \in \R$,  
${K_{a,1}}\left( {s,t} \right)
  = {e^{ - {a^2}{{\left( {s - t} \right)}^2}}}
  = \sum\nolimits_{j = 0}^\infty
  {{\lambda _j}
    {\varphi _j}\left( s \right)
    \overline {{\varphi _j}\left( t \right)} }$, 
where 
eigenvalues 
${\lambda _j} = \sqrt {{{2{v_1}}}/{V}} {B^j}$,  
 eigenfunctions 
${\varphi _j}\left( x \right)
= {e^{ - \left( {{v_3} - {v_1}} \right){x^2}}}
{H_j}\left( {\sqrt {2{v_3}} x} \right)$ with 
physicists' Hermite polynomial 
${H_j}\left( x \right)
= {\left( { - 1} \right)^j}
{e^{{x^2}}}
\frac{{{d^j}}} {{d{x^j}}}{e^{ - {x^2}}}$,  and
the constants in the eigen-expansion are defined as follows, 
\begin{equation}
  \label{eq:eigen-constants}
  v_1^{ - 1} = 4{\xi ^2},
  {v_2} = {a^2},
  {v_3} = \sqrt {v_1^2 + 2{v_1}{v_2}},
  V = {v_1} + {v_2} + {v_3},
  B = {v_2}/V. 
\end{equation}
(See Chapter 4.3 of \cite{rasmussen2006gaussian} for more details). 
The eigenfunctions $\{\varphi_j\}$ form an orthonormal basis
under the inner product
\[{\left\langle {{\varphi _j},{\varphi _k}} \right\rangle _G}
  = \int_\R^{} {{\varphi _j}\left( s \right)
    {\varphi _k}\left( s \right)
    g\left( s \right)ds}
  = {\delta _0}\left( {j - k} \right), \]
where $G $ denotes the univariate Normal distribution $N(0,\xi^2)$
and $g$ is its density function relative to Lebesgue measure.  
The Gaussian measure $G$ corresponds to Assumption \ref{assumption:q}. 

With the univariate expansion, 
$K_{a,\gamma}$ is a tensor product of univariate SE 
kernels and admits the following expansion: 
for $s,t \in \R^{d_n}$, 
\begin{eqnarray}
  \label{eq:eigenexpansion}
    {K_{a,\gamma} }\left( {s,t} \right)
    &=&
        \prod\nolimits_{\{i: \gamma_i=1\} }^{}
        {\left( {\sum\nolimits_{j = 0}^\infty
        {\lambda _j^{\left( i \right)}
        \varphi _j^{\left( i \right)}\left( {{s_i}} \right)
        \overline {\varphi _j^{\left( i \right)}
        \left( {{t_i}} \right)} } } \right)}
        \nonumber \\
    & \equiv &
        \sum\nolimits_{k = 0}^\infty
        {{\mu _k^{\left( \gamma  \right)}}
        \psi _k^{\left( \gamma  \right)}\left( s \right)
        \overline {\psi _k^{\left( \gamma  \right)}\left( t \right)} }, 
  \end{eqnarray}
where $\lambda _j^{(i)}$ is the $j^{th}$ eigenvalue
of $i^{th}$ univariate SE kernel,
$\varphi _j^{(i)}$ is the $j^{th}$ eigenfunction 
of $i^{th}$ univariate SE kernel, 
eigenfunctions $\left\{ {\psi _k^{\left( \gamma  \right)}} \right\} $ and
eigenvalues $\left\{ {\mu _k^{\left( \gamma  \right)}} \right\}$
are ordered by collecting lower order terms first.

The eigenvalue
$\mu_k^{(\gamma)} = (2v_1/V)^{|\gamma|/2} B^m$
for some $m \in \N$ and $m$ is weakly increasing in $k$.
Note the number of $k-$tuples of positive integers whose sum is $m$
is $\binom{m-1}{k-1}$, 
the number of terms involving $B^m$ for $m\ge 1$ is
$\sum\nolimits_{k = 1}^m \binom{m-1}{k-1} \binom{|\gamma|}{k}
=\binom{|\gamma|+m-1}{m}$
by Vandermonde's identity.
By ratio test, the eigenvalues are summable:
$\sum \nolimits_{m=1}^\infty \binom{|\gamma|+m-1}{m} B^m < \infty$. 
As we collect low order terms first, first few terms are 
$\mu_0^{(\gamma)} = (2v_1/V)^{|\gamma|/2}$,
$\mu_k^{(\gamma)} = (2v_1/V)^{|\gamma|/2} B$ for $k=1:|\gamma|$,
$\mu_k^{(\gamma)} = (2v_1/V)^{|\gamma|/2} B^2$ for
$k=|\gamma|+1 : |\gamma|+\binom{|\gamma|+1}{2}$.
(Recall the constants in (\ref{eq:eigen-constants}).)

Eigenfunctions are defined accordingly
depending on eigenvalues.
The ordering of the eigenfunctions with the same eigenvalues 
does not matter. 
The eigenfunctions are orthogonal if the base measure 
is $Q_n \equiv N(0,\xi^2I_{d_n})$ where isotropy is assumed without loss of generality. %

With the eigen-expansion of $K_{a,\gamma}(\cdot, \cdot)$ and the summable
eigenvalues,  
$W^{a, \gamma}$ has the series representation: 
for $Z_j \distras{iid} N(0,1)$, $j=0,1,...,$
\begin{equation}
  \label{eq:series-rep}
W^{a, \gamma}_t = \sum\nolimits_{j=0}^\infty
Z_j \sqrt{\mu^{(\gamma)}_j}\psi _j^{(\gamma)}(t).   
\end{equation}
By orthonormality of $\{\psi_j^{(\gamma)}\}$,
$||W^{a,\gamma}||^2_{L_2(Q_n)} = \sum\nolimits_{j=0}^\infty Z_j^2{\mu^{(\gamma)}_j}$.
Then,   
$W^{a,\gamma} \in {L_2(Q_n)}$, $a.s.$.

\subsection{Concentration Function}
\label{sec:gauss-rand-elem}

With the series representation (\ref{eq:series-rep}),
we can treat the Gaussian process $W^{a,\gamma}$
as a Borel measurable map
into the Banach space $( L_2(Q_n), ||\cdot||_{L_2(Q_n)}) $ such that
the random variable $b^*(W^{a,\gamma}) $ is Gaussian
for every continuous, linear map $b^* : L_2(Q_n) \to \R$. 
 Since $W^{a,\gamma}$ is continuous with probability 1,
 the sample paths under consideration are continuous versions. 
 For any set $U \subset L_2(Q_n)$, 
 the probability 
 $\Pr(W^{a,\gamma} \in U)$ equals the probability 
 $\Pr \left( W^{a,\gamma} \in U \cap C(\R^{d_n})\right)$. 
 In the end, the Banach space where $W^{a,\gamma}$ lives is essentially $\cL_2(Q_n)\equiv L_2(Q_n)\cap C(\R^{d_n}) $.  
 
The RKHS $\cH^{a,\gamma}$
associated with $W^{a,\gamma}$ 
is the completion of the range
$S [\cL_2(Q_n)^*] $ 
where $S[{b^*}]
= \bE [ {{W^{a,\gamma }} b^*( {{W^{a,\gamma }}} )} ]$, 
$b^* \in \cL_2(Q_n)^*$ in the sense of Pettis integral; see \cite{van2008reproducing} for more details. 
The corresponding RKHS norm is determined by the inner product 
${\left\langle {Sb_1^*,Sb_2^*} \right\rangle _{{\cH^{a,\gamma }}}}
= \bE \left( {{b_1^*}\left( {{W^{a,\gamma }}} \right)
    b_2^*\left( {{W^{a,\gamma }}} \right)} \right)$. 
Using the same argument as (2.4) of \cite{van2008reproducing}, 
the RKHS norm is stronger than the $L_2(Q_n)$ norm.  
Hence, $\cH^{a,\gamma}$ is seen as a dense subset of $\cL_2(Q_n)$ and can be isometrically identified with
an $\ell_2$ sequence space. 
In our case, with the eigen-expansion (\ref{eq:eigenexpansion}),  
the RKHS unit ball $\cH_1^{a,\gamma}$ of $W^{a,\gamma}$ 
can be isometrically identified with the following ellipsoid: 
\begin{equation}
  \label{eq:RKHS}
   \left\{ {\left\{ {{\theta _j}} \right\}_{j = 1}^\infty :
      \sum\nolimits_{j = 1}^\infty  {\theta _j^2/{\mu^{\left( \gamma  \right)} _j}}  \le 1} \right\}
  \subseteq {\ell ^2}\left( \bN \right).
\end{equation} 
This isometry provides a different route to compute the metric entropy
of the unit ball of $\cH^{a,\gamma}$.

Prior mass calculations in Theorem \ref{thm:overall-contraction} and
Proposition \ref{proposition:FP-prior-prob} require both upper and lower 
bounds for $\Pr(\|W^{a,\gamma} - f_n^*\|_{L_2(Q_n)} < \rateTwo_n)$, 
the $f^*_n$-shifted $\rateTwo_n$-ball probability of ${W^{a,\gamma}}$.


With the above abstract formulation,
the small ball probability, in log scale, can be bounded as
\begin{equation}
  \label{eq:concentration-inequality}
  {\phi^{a,\gamma} _{f_n^*}}({{\varepsilon _n}} )
  \le
  - \log \Pr ( {\|W^{a,\gamma} - {f_n^*}{\|_{L_2(Q_n)}}
      < {\varepsilon _n}})
  \le
  {\phi^{a,\gamma} _{f_n^*}}( {{\varepsilon _n}/2} ),
\end{equation}
with $\phi^{a,\gamma}_{f^*_n}$ denoting the concentration function
\begin{equation}
  \label{eq:concentration-fcn}
  \phi _{{f_n^*}}^{a,\gamma }(\rateTwo_n)
  = \mathop {\inf }\limits_{h \in {\cH^{a,\gamma }}:
    \|h - f_n^*\|_{L_2(Q_n)} < \rateTwo_n}
  \|h\|^2_{\cH^{a,\gamma}}
  - \log \Pr (\|W^{a,\gamma}\|_{L_2(Q_n)} < \rateTwo_n), 
\end{equation}
where $\|\cdot\|_{\cH^{a,\gamma}}$ is the canonical norm of the Hilbert space $\cH^{a,\gamma}$. 

With inequality (\ref{eq:concentration-inequality}),
bounding shifted small ball probability is essentially bounding
the concentration function. 
The concentration function has two parts: 
the decentering part and the centered small ball probability exponent.
The decentering part measures the position of the centering $f_n^*$
relative to the RKHS.

\subsection{Centered Small Ball Probability Bounds via Metric Entropy}
\label{sec:entr-calc-small}

The centered small ball probability 
is bounded using the metric entropy method
\citep{kuelbs1993metric,li1999approximation, van2008reproducing}.
The metric entropy method links the metric entropy of the RKHS unit ball 
with the centered small ball probability. 
Section 6 of the review paper \cite{van2008reproducing} summarizes the quantitative relationship between bounds of metric entropy and bounds of small ball probability. 
In our analysis, we first calculate the metric entropy of the RKHS unit ball and
then use the relationship to derive bounds of 
centered small ball probabilities as a corollary.

With isometry, the metric entropy of the RKHS unit ball is
the metric entropy of the $\ell_2$ ellipsoid (\ref{eq:RKHS}).
It is well known that the metric entropy of $\ell_2$ ellipsoid
in the fashion of (\ref{eq:RKHS}) depends on the decay rate of $\{\mu_j^{(\gamma)}\}$.
Lemma \ref{lemma:metric-entropy} gives bounds for the metric entropy of the RKHS unit ball.

\begin{lemma}
  \label{lemma:metric-entropy} 
  Suppose $\cH_1^{a,\gamma}$ is the RKHS unit ball
  associated to the GP $\segp (\R^{d_n};\gamma, \delta_ a)$ with 
  the design measure $Q_n \equiv N(0,\xi^2I_{d_n} )$, 
  constants $v_1$ and $V$ are defined in (\ref{eq:eigen-constants}), 
 constants $a$,  ${\varepsilon}$ and $|\gamma|$ 
  satisfy $\varepsilon^{-2}  \ge C_{H}{\left( {a \xi } \right)^{|\gamma|}}$ 
  for some constant $C_{H}$ such that
  ${\log \left( {\frac{1}{\varepsilon} } \right)
    - \frac{|\gamma|}{4}
    \log \left( {\frac{V}{{2{v_1}}}} \right)}
  \asymp \log \left( {\frac{1}{\varepsilon} } \right) $,  
  and ${a \xi } \log \left( {1/\varepsilon } \right) > |\gamma|$,
  then the metric entropy of $\cH^{a,\gamma}_1$ satisfies
  \[\begin{array}{l}
      \log N\left( {\cH_1^{a,\gamma},\varepsilon ,\| \cdot |{|_{L_2(Q_n)}}} \right)
      \precsim
      a ^{|\gamma|}
      \log {\left( {1/\varepsilon } \right)^{|\gamma| + 1}}/|\gamma|!,  
      \\ 
      \log N\left( {\cH_1^{a,\gamma},\varepsilon ,\| \cdot |{|_{L_2(Q_n)}}} \right)
      \succsim
      a ^{|\gamma|}
      \log {\left( {1/\varepsilon } \right)^{|\gamma|}}/|\gamma|!. 
    \end{array}\]
\end{lemma}

The proof of Lemma \ref{lemma:metric-entropy} follows standard technique of metric entropy of $\ell_2$ sequence spaces. 
The lower bound is smaller than the upper bound by a logarithmic factor.
This gap affects the bounds for small ball probabilities,
but it does not jeopardize 
Theorem \ref{thm:overall-contraction} and Proposition \ref{proposition:FP-prior-prob}. 

Then the centered small ball probability bounds are obtained in Lemma \ref{lemma:centered-sbp} as a corollary. 
\begin{lemma}
  \label{lemma:centered-sbp}
  Suppose $\phi _0^{a,\gamma }$ is the concentration function
  associated to the GP $\segp (\R^{d_n};\gamma, \delta_ a)$ with 
  the design measure $Q_n \equiv N(0,\xi^2I_{d_n} )$, 
  constants $a$ and  ${\varepsilon}$ 
  satisfy $\varepsilon^{-2}  \ge C_{H}{\left( {a \xi } \right)^{|\gamma|}}$ and
  ${a \xi } \log \left( {1/\varepsilon } \right) > |\gamma|$, 
  then there exists a constant $C$ independent of $a,\xi$ and $|\gamma|$ 
  such that  
  \[\phi _0^{a,\gamma }\left( {{\varepsilon}} \right)
    \le
    C a ^{|\gamma|}
    \log {\left( {a/\varepsilon } \right)^{|\gamma| + 1}}/|\gamma|!, 
  \]
  and there exists a constant $C^\prime$ such that
  \[
    \phi _0^{a,\gamma }\left( {{\varepsilon }} \right)
    \ge
    C^\prime {a^{|\gamma|}}
    \log {\left( {1/\varepsilon } \right)^{|\gamma|}}/|\gamma|!. 
  \]

\end{lemma}

\begin{proof}
  With the assumptions, the metric entropy calculation in Lemma \ref{lemma:metric-entropy} holds.
  Following the same idea as in Lemma 4.6 of \cite{van2009adaptive},
  the first assertion holds.
  
  Using the first inequality of Lemma 2.1 in \cite{aurzada2009small},
  let $\lambda=2$ in the Lemma, 
  $\phi _0^{a,\gamma }\left( {{\varepsilon _n}} \right)
  \ge       
  \log N\left( {\cH^{a,\gamma},\varepsilon_n ,\| \cdot \|_{L_2(Q_n)}} \right)-2
  \ge 
  C^\prime {a^{|\gamma|}}
  \log {\left( {1/\varepsilon } \right)^{|\gamma|}}/|\gamma|!$ 
  holds for some constant $C^\prime$.
  
\end{proof}

\subsection{Shifted Small Ball Probability Estimates} 
\label{sec:decentering}
With centered small ball probability calculation, 
bounds for shifted small ball probability are readily available if
the bounds of the decentering part are obtained. 
Lemma \ref{lemma:decentering-upper-bd} and Lemma \ref{lemma:decentering} 
give upper and lower bounds of the decentering part
as a function of rescaling level and model index parameter.  
Lemma \ref{lemma:decentering-upper-bd} is used to \textit{lower} bound the prior mass on
$ \{ \|W^{a,\gamma_n^*} - {f_n^*}\|_{L_2(Q_n)} < \varepsilon_n  \}$; 
 Lemma \ref{lemma:decentering}  is used to \textit{upper} bound the prior mass on
 $\{\Gamma \in \fp (\gamma_n^*)\} \cap 
 \{ \|W^{a,\Gamma} - {f_n^*}\|_{L_2(Q_n)} < \varepsilon_n  \}$.

\begin{lemma} 
  \label{lemma:decentering-upper-bd}
  Suppose
  $\cH^{a,\gamma_0}$ is the RKHS 
  of the GP $\segp (\R^{d_n};\gamma_n^*, \delta_ a)$ with 
  the design measure $Q_n \equiv N(0,\xi^2I_{d_n} )$,
  $\gamma_n^*\in \{0,1\}^{d_n}$ encodes the indices of true variables, 
  $d_0 \equiv |\gamma_n^*|$, 
  $f_0 \in H^{\beta}(\R^{d_0}) \cap {L_2(Q_0)} $,
  $f_n^* = T_{\gamma_n^*} f_0 $,
   then for every $a>0$, 
  there exists a constant $C$, $\varepsilon_0$
  such that for all
  $\varepsilon < \varepsilon_0$
    \[\mathop {\inf }
      \limits_{h \in {\cH^{a,\gamma_0 }}:
        \|h - {{f}_n^*}\|_{L_2(Q_n)} < {\varepsilon }}
    \|h\|^2_{{\cH^{a,\gamma_0 }}}
    \le
    C     {\left( {2\sqrt \pi  } \right)^{d_0}} 
    {a^{d_0}} 
    {e^{C{\varepsilon }^{ - 2/\beta }/{a^2}}}. 
  \]
\end{lemma}

The proof of
Lemma \ref{lemma:decentering-upper-bd} 
leverages the representation theorem of the RKHS
and uses the squared RKHS norm of a special element
in the true model neighborhood as an upper bound. 
In light of Lemma 4.1 of \cite{van2009adaptive} and
Lemma 7.1 of \cite{van2008reproducing},
after embedded into $L_2(Q_n)$, 
elements in $\cH^{a,\gamma}$ admit the representation:
for $t \in \R^{d_n}$, 
${h_\psi }\left( t \right)
  =
  \int_{{\R^{|\gamma|} }}
  {{e^{i\left( {\lambda ,t_0} \right)}}
    \psi \left( \lambda  \right)
    {m_{a,\gamma }}\left( \lambda  \right)d\lambda }$, 
  where $t \equiv (t_0,t_1)$ with
  $t_0 \in \R^{|\gamma|}$ and $t_1 \in \R^{d_n} \backslash \R^{|\gamma|}$,
  $m_{a,\gamma}(\cdot)$ is the spectral density of the
  $\gamma-$dimensional Gaussian process with rescaling level $a$, 
  and
  $\psi $ is in the complex Hilbert space $ L_2(m_{a,\gamma})$;
  its  RKHS norm is defined as 
  $\|{h_\psi }\|_{{\cH^{a,\gamma }}}^2
    = \int_{{\R^{|\gamma|} }}^{} {|\psi \left( \lambda  \right){|^2}
      {m_{a,\gamma }}\left( \lambda  \right)d\lambda }$ .
  The spectral density 
  ${m_{a,\gamma }}\left( \lambda  \right) 
    = {a^{ - |\gamma|}}{m_\gamma }\left( {\lambda /a} \right)$,  
  where ${m_\gamma }\left( \lambda  \right) 
  = {e^{ - \|\lambda \|_2^2/4}}/\left( {{2^{|\gamma|}}{\pi ^{|\gamma|/2}}} \right)$
  and $\lambda \in \R^{|\gamma|} $.

\begin{lemma}
  \label{lemma:decentering}
  Suppose 
  $\cH^{a,\gamma}$ is the RKHS 
  of the GP $\segp (\R^{d_n};\gamma, \delta_ a)$ with 
  the design measure $Q_n \equiv N(0,\xi^2I_{d_n} )$,
  $\gamma_n^*\in \{0,1\}^{d_n}$ encodes the indices of true variables, 
  $d_0 \equiv |\gamma_n^*|$, 
  $f_0 \in H^{\beta}(\R^{d_0}) \cap L_2(Q_0)$ with  
  Fourier transform satisfying 
  $|\widehat {{{f}_0}\sqrt {g_{d_0}} }\left( \lambda  \right)| 
  \succsim \|\lambda\|^{-(\alpha+d_0/2)}_2 $ 
  for all $\|\lambda\|_2 \ge 1$ and some constant $\alpha>0$,
  $f_n^* = T_{\gamma_n^*} f_0$,
  and $c_\xi \equiv \xi / \sqrt{2}$.  
  Pick a $\gamma \in  \fp(\gamma_n^*)$, 
  then for all $a>0$, 
  there exist constants $C$, $C^\prime$ and $\varepsilon_0$
  only dependent of $f_0$, such that  for all
  $\varepsilon < \varepsilon_0$, 
  \[\mathop {\inf }
    \limits_{h \in {\cH^{a,\gamma }}: 
      \|h - {f}_n^* \|_{L_2(Q_n)} < {\varepsilon}}
    \|h\|^2_{{\cH^{a,\gamma }}}
    \ge
    C {\varepsilon ^2}
    { (c_\xi a)^{|\gamma |}}
    {e^{C^\prime {\varepsilon^{-2/\alpha}}\left( {{\xi ^2}
            \wedge {a^{ - 2}}} \right)}}. 
  \]  
\end{lemma}


The strategy of the proof follows the proof of Theorem 8 in
\cite{vaart2011information}.
In Lemma \ref{lemma:decentering}, 
the RKHS norm is expected to explode as the neighborhood in
the false positive model space shrinks.
Also note the effect of rescaling level $a$:
larger $a$ means better approximation and hence larger RKHS norm;
small $a$ means flatness in the prior sample path and hence smaller RKHS norm. 
\relax

\section{Discussion} 
\label{sec:discussion} 
We have shown here that a GP regression model equipped with stochastic variable selection can simultaneously offer adaptive regression function estimation and consistent recovery of its sparsity pattern. This result is derived under several assumptions, some of which are mathematical formalizations of reasonable statistical considerations while others are needed more for technical reasons than statistical ones. Below we offer a detailed discussion of the reasonability and limitations of the formal assumptions and explore possible relaxations. 

\subsection{Gaussian design assumption}
Assumption \ref{assumption:q} requires a Gaussian random design and 
is quite restrictive, but is needed for \h{a specific technical reason: it leads to a Karhunen-Lo\`eve eigen expansion of
the squared exponential kernel in closed form}. 
With the Karhunen-Lo\`eve expansion,  
the RKHS unit ball is isometric to an $\ell_2$ sequence space
whose metric entropy can be accurately bounded. 
Further, small ball probability estimates are available via the metric entropy method
\citep{kuelbs1993metric,li1999approximation, van2008reproducing}. 
In particular, it permits an upper bound of the prior mass condition that
is crucial for Proposition \ref{proposition:FP-prior-prob}. 
Also, the assumption $\xi^2 >2/e$ is made to simplify the calculations for
Proposition \ref{proposition:FP-prior-prob}. 
It holds in most applications as the design matrix can be standardized. 

Without the Gaussian assumption, it is not tractable to work out 
sharp upper bounds of the prior mass condition needed for 
 Proposition \ref{proposition:FP-prior-prob}.   
An alternative approach to metric entropy calculation of the RKHS unit ball 
without assuming Gaussian random design  
is to \h{extend \cite{aurzada2009small} where the RKHS unit ball} is identified as a set of 
``well-behaved'' entire functions whose metric entropy can be accurately bounded. 
However, direct extension carrying over the rescaling parameter gives 
sub-optimal lower bounds of the metric entropy, and the resulting 
upper bounds of the prior mass condition become meaningless.

A natural relaxation is to assume the Radon-Nikodym derivative of $Q_n$ with respect to $G_{d_n}$ is bounded away from 0 and $\infty$, uniformly across all $n \ge 1$.  This uniform absolute continuity implies that convergence in $\|\cdot\|_{L_2(Q_n)}$ is equivalent to convergence in $\|\cdot\|_{L_2(G_{d_n})}$. 
Clearly, the uniform boundedness assumption renders the relaxation quite limited.

\subsection{Fixed sparsity assumption}

Assumption \ref{assumption:f0} makes it clear that for all large $n$,
the true function is sparse and the support size $d_0$ does not grow with sample size. While the latter condition may appear too restrictive, it is shown in \cite{yang2015minimax} that in the case $d_n$ grows nearly exponentially in $n$, one cannot hope to consistently estimate a sparse, smooth regression function nonparametrically unless the true support size remains essentially constant.

Assumption \ref{assumption:sparsity} identifies  the true sparsity level $d_0$ 
as the minimal support size for the sequence $(f^*_n: n \ge 1)$.
Under this assumption, each of the $d_0$ inputs to $f_0$ results in a variability that is detectable in the $L_2$ topology.
This is essentially a nonparametric version of the $\beta$-min assumption in 
sparse linear regression literature.  

\subsection{Limited smoothness assumption}
Assumption \ref{assumption:smoothness} imparts only a limited amount of smoothness on $f_0$.
The first part of the assumption requires true $f_0$ to be in the Sobolev space  $H^\beta(\R^{d_0})$ in which functions satisfy  (\ref{eq:sobolev-norm}).
Analogous to \Holder smoothness, functions in $H^\beta(\R^{d_0})$ have square-integrable (weak) partial derivatives up to order $\beta$. The second part of Assumption \ref{assumption:smoothness} is adapted from \cite{vaart2011information}
and combines the probability density of the random design. 
It encodes a lower bound for the regularity of $f_0$ 
 in the spirit of the self-similarity assumption in 
 \cite{gine2010confidence, bull2012honest, szabo2015frequentist, ray2017adaptive}.
 A direct consequence of the assumption is lower bounds for the RKHS norm of the functions
 in a $L_2$ neighborhood of the true function   
 (Lemma \ref{lemma:decentering}).
 Further, the resulting lower bounds are necessary for upper
 bounding the prior mass condition in Proposition \ref{proposition:FP-prior-prob}. 
The self-similarity assumption in the Gaussian sequence model can be
written as the convolution of $\hat{f}_0$ and a sum of delta functions
evaluated at some frequencies. 
This convolution reflects the fixed design nature of Gaussian sequence model. 
In contrast, with random design, 
it is reasonable to weigh ``signals'' at different frequencies differently. 

Note that the decay rate of a convolution is bounded above by the sum of the decay rates of the functions convolved when they are in both $L_1$ and $L_\infty$.  As $\widehat{\sqrt{g_{d_0}}}$ is a Gaussian function decaying exponentially, $\widehat{\sqrt{g_{d_0}}} \in L_1 \cap L_\infty$. Part 1 implies $\hat{f}_0\in L_1 \cap L_\infty$. Therefore, part 2, coupled with part 1, requires $|\widehat {{f_0}\sqrt {g_{d_0}} }\left( \lambda  \right)|$ to decay approximately at the same rate as $|\hat{f}_0(\lambda)|$ as $\|\lambda\| \to \infty$, that is, $|\hat{f}_0(\lambda)| \succsim \|\lambda\|^{-(\alpha+d_0/2)}_2$ as $\|\lambda\| \to \infty$. 

One concern is whether there exists an $f_0$ satisfying Assumption \ref{assumption:smoothness}. The answer is yes. As an example, let $\hat f_0(\lambda) = (1 + \|\lambda\|)^{-r}$ with $r = \alpha + d_0/2 > \beta + d_0/2$. Then $f_0 \in H^\beta(\R^{d_0})$ and $\widehat{f_0 \sqrt{g_{d_0}}}(\lambda) = \hat f_0 * \widehat{\sqrt{g_{d_0}}} (\lambda) \ge \int_{\|t\| \le 1} \hat f_0(\lambda - t) \widehat{\sqrt{g_{d_0}}}(t) dt \succsim ( 2+ \|\lambda\|)^{-r} \asymp \|\lambda\|_2^{-(\alpha + d_0/2)}$. 
More such functions could be constructed as long as the tail decay rate of $\hat f_0$ is maintained. 

The limited smoothness assumption rules out true functions that are infinitely differentiable. This assumption is key to our proof strategy which exploits a polynomial slowdown in posterior contraction rate when spurious predictors are included in the model. Similar rate slowdown also manifests for infinitely smooth true functions, but the depreciation is only logarithmic \citep{van2009adaptive}. While our proof strategy fails to exploit such subtler rate drops, variable selection consistency may still hold in these cases. Indeed, for Bayesian linear regression with variable selection, posterior contraction rates slow down only by a multiplicative constant when spurious variables are included, and yet variable selection consistency holds in such cases. 

Curiously, the results presented in this paper could be applied to construct a modified nonparametric regression model with guaranteed variable selection consistency, no matter the level of smoothness of the true function. One can augment the design matrix with an additional, independent Gaussian ``covariate'' $Z$, and create a new response $Y_i^\prime = Y_i + g_0(Z_i)$ where $g_0$ is a known function of limited smoothness. 
Now, for the new design $(X,Z)$ and the modified response data $Y^\prime$, the underlying true function has limited smoothness, and hence our results guarantee asymptotically accurate recovery of the important coordinate variables of $X$ in addition to the the synthetic variable $Z$. However, the resulting posterior convergence rate may be slower than optimal if $f_0$ were smoother than $g_0$.

\subsection{Restricted design dimension growth assumption}

Assumption \ref{assumption:d} restricts the applicability of our result only up to a design size 
$\log (d_n) \precsim n^{d_0/(2\beta + d_0)}$. Consequently, we are restricted to $\log(d_n) \ll n$ when $\beta$ is large. This is a serious limitation because as shown in \citep{comminges2012tight}, with true sparsity fixed at $d_0$, variable selection consistency should hold for larger design sizes, up to the limit $\log(d_n) = O(n)$. Notice however that we prove results for a Bayesian estimation method that simultaneously infers the sparsity patter $\Gamma$ and the regression function $f$, and offers a near minimax optimal estimation of the latter by adapting to the true smoothness level. It is this adaptation that imposes the stricter bound on $d_n$, beyond which the estimation error rate is dominated by the variable selection penalty which does not worsen polynomially between correct selection and false positive selections.

\h{It is unclear at the moment whether the Bayesian estimation model studied here, or any other model which offers smoothness adaptive function estimation, could actually achieve variable selection consistency with $d_n$ growing faster than our bound. However, if one were to sacrifice on adaptive function estimation, variable selection consistency may be achieved even when $n^{d_0/(2\beta  + d_0)} \ll \log (d_n) \precsim n$ under a variation of our GP regression model where the random rescaling component (Assumption \ref{assumption:A}) is replaced with a dimension-specific deterministic rescaling
\[
\pi_{n,d}(a) = \mathds{1}(a = n^{1/(2\underline\beta + d)}),
\]
for every $n \in \N$, $d \in \{1, \ldots, d_n\}$, where $\underline \beta$ is a fixed, small positive scalar.
This modification grants variable selection consistency up to design dimensions $d_n$ with $\log (d_n) \precsim n^{d_0 / (2\underline\beta + d_0)}$ as long as the true smoothness $\beta > \underline \beta$.
However, the posterior contraction rate under the deterministically rescaled GP prior is $O(n^{ -\underline\beta/(2\underline\beta + d_0)})$ (up to a logarithmic factor of $n$), which could be much slower than the optimal rate $n^{ -\beta/(2\beta + d_0)}$. That is, in the extreme case $\log (d_n) \precsim n$, we can pick a small positive $\underline{\beta}$ for the deterministic rescaling to guarantee variable
selection consistency at the cost of estimation accuracy. 
}

\subsection{Separating variable selection from estimation} 
The discussion in the above two subsections indicates that one could gain on variable selection consistency by sacrificing on optimal function estimation. This leads to a much broader question. Should the two inference goals, optimal function estimation and consistent variable selection, be separated, and perhaps approached in a two-stage manner? We believe the answer is {\it yes}. A two-stage approach may sound contradictory to our adopted position of Bayesian modeling and inference. But in reality, the Bayesian paradigm, when augmented with decision theoretic considerations, may offer an incredibly fertile ground for exploring a rigorous two-stage approach. For example, one may adapt \cite{hahn2015decoupling} and estimate $\gamma$ by the formal Bayes estimate $\hat \gamma_B$ under the loss function $L((f,\gamma), (\hat f, \hat \gamma))=\|f - \hat f\|^2_{L_2(Q_n)} + \lambda J(\hat\gamma)$, for some model complexity penalty function $J(\gamma)$ and penalty tuning parameter $\lambda > 0$. The Bayes estimate may be computed as $\hat\gamma_B = \arg\min_{\gamma}\{\|\bar f - \bar f_{\gamma} \|_{L_2(Q_n)}^2 + \lambda J(|\gamma |)\}$
where $\bar f$ is the posterior mean of $f$ and
$\bar f_{\gamma} = \int \bar f(x) \prod_{j:  \gamma_j = 0}G_1(dx_j)$
is the projection of $\bar f$ along $\gamma$
under $L_2(Q_n)$. Whether such computations are feasible and result in consistent estimation of $\gamma$ remain to be seen. It seems likely that the additional penalty component may help resolve the true model from false positive models, even if the posterior weights assigned to the latter were not polynomially smaller. 
      
      Besides potential theoretical gains, such a Bayesian decision theoretic approach is appealing on philosophical grounds alone. Any statistical analysis may have any number of inference goals and no single estimation method may be universally optimal. For example, cross-validation based LASSO is known to offer great prediction accuracy while suffering from enhanced false detection \cite{bogdan2015slope}. A Bayesian decision theoretic approach may help unify such disparate goals where the analyst fits a single model to the data, but produces different estimates for different quantities by utilizing appropriately chosen loss function for each task. For example, in the adaptation described above, one may still report the posterior mean $\bar f$ as the Bayes estimate of $f$ (under the integrated square loss) while producing a sparse estimate $\hat \gamma_B$ for variable selection according the loss function discussed above. It will be exciting to carry out rigorous posterior asymptotic behavior of such formal Bayes estimates.

\bibliographystyle{imsart-nameyear}
\bibliography{GP-VS-consistency-bib}

\begin{supplement}
\textbf{Supplement to ``Variable selection consistency of Gaussian process regression''}.
This Supplement contains additional results and proofs in the text. 
\end{supplement}

\newpage

\section{Supplement to ``Variable selection consistency of Gaussian process regression''}
\label{sec:ss}

The supplement file contains complete proofs for
Theorem  \ref{thm:overall-contraction},
Proposition \ref{proposition:FP-prior-prob},
Lemma   \ref{lemma:metric-entropy},
Lemma \ref{lemma:decentering-upper-bd},
Lemma   \ref{lemma:decentering} in the paper
``Variable selection consistency of Gaussian process regression''.

\subsection{Proof of Theorem  \ref{thm:overall-contraction}}
\label{sec:overall-contraction-proof}

\begin{proof}
  It suffices to show there exist sets (sieve) $\bB_n$, such that the following three conditions hold for all sufficiently large $n$:
  \[\begin{array}{l}
      \Pi_n \left( {\|{W^{A,\Gamma }} - {f_n^*}\|_{L_2(Q_n)} \le {\varepsilon _n}} \right)
      \ge {e^{ - n\varepsilon _n^2}}\\
      \Pi_n \left( {{W^{A,\Gamma }} \notin {\bB_n}} \right)
      \le {e^{ - 4n\varepsilon _n^2}}\\
      \log N\left( {{\varepsilon _n},{\bB_n},\| \cdot \|_{L_2(Q_n)}} \right)
      \le n\varepsilon _n^2 . 
    \end{array}\]


  \textbf{Prior mass condition}

  Assumption $\beta > d_0/2$ implies
  $1/{\varepsilon _n} \ge {C_H}{\left( {a\xi } \right)^{d_0/2}}$
  and $a \xi \log (1/\varepsilon_n)>d_0$
  hold for every $a \in [C(1/\varepsilon_n)^{1/\beta},2 C(1/\varepsilon_n)^{1/\beta}]$
  with some constant $C$. 
  We can apply Lemma \ref{lemma:centered-sbp} and  
  Lemma \ref{lemma:decentering-upper-bd} to obtain the following lower bound 
  \[\begin{array}{lll} 
      && \Pi_n  \left( {\|{W^{A,{\gamma _n^*}}} - {f_n^*}\|_{L_2(Q_n)}
         \le 2 {\varepsilon _n}} \right)
      \\
      &\ge& \int_{K_n}
            ^{2K_n }
            {\Pi_n \left( {\|{W^{a,{\gamma _n^*}}} - {f_n^*}\|_{L_2(Q_n)}
            \le 2{\varepsilon _n}} |a \right)
            \Pi_n\left( da \right)} \\
      &\ge&
            \int_{K_n } ^{2K_n}
            {{e^{ - \phi _{{f_n^*}}^{a,{\gamma _n^*}}\left( { {\varepsilon _n}} \right)}}
            \pi_n \left( a \right)da} \\
      &\ge&
            \int_{K_n} ^{2K_n}
            {{e^{
            - \left( {C{a^{d_0}}
            + {C^\prime}{a^{d_0}}
            \log {{\left( {a/{\varepsilon _n}} \right)}^{d_0 + 1}}} \right)}}
            \pi_n\left( a \right)da} \\
      &\ge&
            {e^{ - C{\varepsilon ^{ - d_0/\beta }}
            \log {{\left( {1/{\varepsilon _n}} \right)}^{d_0 + 1}}}}  
    \end{array}\]    
  where $K_n= C{{\left( {1/{\varepsilon _n}} \right)}^{1/\beta }}$, 
  ${\varepsilon _n} \asymp  {n^{ - 1/\left( {2 + {d_0}/\beta } \right)}}
  (\log n) ^{\kappa_1} $ and  
  ${\kappa _1} = \frac{d_0+1}{ {2 + {d_0}/\beta } }$
  such that quantity in  the exponent satisfies 
  ${\varepsilon_n ^{ - d_0/\beta }}
  \log {{\left( {1/{\varepsilon _n}} \right)}^{d_0 + 1}}
  \precsim n\varepsilon _n^2$. 
  We can achieve
  \[{\Pi _n}
    \left( {\|{W^{A,{\gamma _n^*}}} - {f_n^*}\|_{L_2(Q_n)}
        \le {\varepsilon _n}} \right)
    \ge
    {e^{ - \frac{1}{2}n\varepsilon _n^2}}\]
  by choosing ${\varepsilon _n}$ to be a large multiple of 
  ${n^{ - 1/\left( {2 + {d_0}/\beta } \right)}}
  (\log n) ^{\kappa_1}$.
 
  Therefore, with prior mass of model $\gamma_n^*$ satisfying   
  $\Pi_n \left( {\Gamma  = {\gamma _n^*}} \right) 
  = \Pi_n (|\Gamma| = d_0) \binom{d_n}{d_0}^{-1} 
  \ge 
  {e^{ - \frac{1}{2}n\varepsilon _n^2}}$,  
  \[\begin{array}{lll}
     && \Pi_n \left( {\|{W^{A,\Gamma }} - {f_n^*}\|_{L_2(Q_n)}
      \le {\varepsilon _n}} \right) \\ 
      &\ge& \Pi_n \left( {\Gamma  = {\gamma _n^*}} \right)
            \Pi_n \left( {\|{W^{A,{\gamma _n^*}}} - {f_n^*}\|_{L_2(Q_n)}
            \le {\varepsilon _n}} \right)\\
      &\ge& {e^{ - n\varepsilon _n^2}}.   
\end{array}\]

\textbf{Sieve construction}
  
  The sieve $\{\bB_n\}$ is constructed as
  \[{\bB_n} =
    \bigcup\nolimits_{\gamma  \in {{\left\{ {0,1} \right\}}^{d_n}}:
      |\gamma|\le \underline{d}_n }
    { {\bB_{n,\gamma }}} \]
  where $ \underline{d}_n = C{\left( {n\varepsilon _n^2} \right)^{1/\rho }}$
  for some constant $C$ 
  and $\bB_{n,0^{d_n}} = [-M_n,M_n]$; 
  for $\gamma \neq {0}^{d_n}$, 
  \[{\bB_{n,\gamma }}
  =
    {M_n}\sqrt {{r_n}} \cH_1^{r_n,\gamma }
    + {\varepsilon _n}{\bB_{1,\gamma }}  , \] 
where ${\bB_{1,\gamma }}$ is the unit ball in the Banach space $T_\gamma^{d_n} L_2(\R^{|\gamma|})$
indexed by $\gamma$.
  $M_n$, and $r_n$ 
  are specified such that 
  $M_n^2 \asymp  n\varepsilon _n^2$, and 
  $r_n^{|\gamma |}
  { {\log ^{|\gamma | + 1} \left( n \right)} }
  \asymp 
  n\varepsilon _n^2 $. 
  The choice of $r_n$ depends on $\gamma$ 
  but for ease of notation $\gamma$ is dropped. 
  To apply Lemma \ref{lemma:metric-entropy} and Lemma \ref{lemma:centered-sbp}, it requires
  $r_n^{|\gamma |} \precsim \varepsilon _n^{ - 2}$.
  Clearly, the choice of $r_n$ satisfies this requirement.

  \textbf{Verifying condition
    $\Pi_n \left( {{W^{A,\Gamma }} \notin {\bB_n}} \right)
    \le {e^{ - 4n\varepsilon _n^2}}$} 

  For $\gamma = 0^{d_n}$, the prior on the regression function is $N(0,1)$, 
  \[\begin{array}{lll}
      \Pi_n \left( {{W^{{0^{d_n}}}} \notin {\bB_n}} \right) 
      &\le & \Pi_n \left( {{W^{{0^{d_n} }}} \notin {\bB_{n,{0^d}}}} \right)\\
      &= & 2\left( {1 - \Phi \left( {{M_n}} \right)} \right)\\
      &\le& \frac{2}{{\sqrt {2\pi } {M_n}}}{e^{ - M_n^2/2}} 
    \end{array}\]
  where $\Phi$ denotes standard Normal cdf.
  By choosing $M_n$ to be a large multiple of $n\varepsilon _n^2$, 
  $\Pi_n \left( {{W^{{0^{d_n} }}} \notin {\bB_n}} \right)
  \le {e^{ - 4n\varepsilon _n^2}}$ holds for all sufficiently large $n$.

  In light of Lemma 4.7 of \cite{van2009adaptive},
  the nesting property 
  \[{M_n}\cH_1^{a,\gamma } + {\varepsilon _n}{\bB_{1,\gamma }}
    \subseteq {B_{n,\gamma }}\] 
  holds for every $a\in [1/\xi, r_n]$.
  By Borell's inequality, for every $a \in [1/\xi,r_n]$ and $\gamma \neq 0^{d_n} $, 
  \[\begin{array}{lll}
      \Pi_n \left( {{W^{a,\gamma }} \notin {\bB_n}} \right)
      &\le&
            \Pi_n \left( {{W^{a,\gamma }} \notin {\bB_{n,\gamma }}} \right)\\
      & \le&
             \Pi_n \left( {{W^{a,\gamma }} \notin {M_n}\cH_1^{a,\gamma }
             + {\varepsilon _n}{\bB_{1,\gamma }}} \right)\\
      &\le&
            1 - \Phi \left(
            {{\Phi ^{ - 1}}\left(
            {{e^{ - \phi _0^{a,\gamma }\left( {{\varepsilon _n}} \right)}}} \right)
            + {M_n}} \right)\\
      &\le&
            1 - \Phi \left( {{\Phi ^{ - 1}}\left( {{e^{ - \phi _0^{{r_n},\gamma }\left( {{\varepsilon _n}} \right)}}} \right) + {M_n}} \right)
    \end{array}\]
  where last inequality is because 
  ${e^{ - \phi _0^{a,\gamma }\left( {{\varepsilon _n}} \right)}}
  = \Pi_n \left( {\|{W^{a,\gamma }}\|_{L_2(Q_n)} \le {\varepsilon _n}} \right)$
  is decreasing in $a$.

  In light of Lemma \ref{lemma:centered-sbp},
  $ {M_n} \ge 4\sqrt {\phi _0^{{r_n},\gamma }\left( {{\varepsilon _n}} \right)} $
  holds for all sufficiently large $n$. 
  Since ${e^{ - \phi _0^{{r_n},\gamma }\left( {{\varepsilon _n}} \right)}} < 1/4$ holds
  for all small enough $\varepsilon_n$,
  then it follows 
  ${M_n} \ge
  - 2{\Phi ^{ - 1}}\left( {{e^{ - \phi _0^{{r_n},\gamma }\left( {{\varepsilon _n}} \right)}}} \right)$ and
  the above inequality is further upper bounded by
  \[1 - \Phi \left( {{M_n}/2} \right) \le {e^{ - M_n^2/8}}.\] 
  So for every $\gamma \in \{0,1\}^{d_n}  \backslash {0}^{d_n} $, the following  
  \[\begin{array}{lll}
      && \Pi_n \left( {{W^{A,\gamma }} \notin {\bB_n}} \right) \\
      &\le&
            \Pi_n \left( {{W^{A,\gamma }} \notin {\bB_{n,\gamma }}} \right)\\
      &\le&
            \int_{1/\xi}^{{r_n}}
            {\Pi_n \left( {{W^{a,\gamma }} \notin {M_n}\cH_1^{a,\gamma }
            + {\varepsilon _n}{\bB_{1,\gamma }}} \right)
            \pi_n \left( {a|\gamma } \right)da}
            + \Pi_n \left( {A > {r_n}|\gamma } \right)\\
      &\le&
            {e^{ - {M_n^2}/8}}
            + {e^{ - C_2 r_n^{|\gamma |}\log^{|\gamma|+1} {\left( {{r_n}} \right)}
            + C_3 \log(|\gamma|)}}\\
      &\le& \frac{1}{2}{e^{ - 4n\varepsilon _n^2}}
    \end{array}\]
  holds for all sufficiently large $n$.
  Hence,
  \[\begin{array}{lll}
      \Pi_n \left( {{W^{A,\Gamma }} \notin {\bB_n}} \right)
      &=&
          \sum\nolimits_{\gamma  \in {{\left\{ {0,1} \right\}}^d}}
          {\Pi_n \left( {{W^{A,\gamma }} \notin {\bB_n}} \right)
          \Pi_n \left( {\Gamma  = \gamma } \right)} \\
      &\le & \frac{1}{2}{e^{ - 4n\varepsilon _n^2}} 
             +
             \sum\nolimits_{\gamma  \in {{\left\{ {0,1} \right\}}^{d_n} }:
             |\gamma|> \underline{d}_n}
          \Pi_n \left( {\Gamma  = \gamma } \right) \\
      &\le &
              {e^{ - 4n\varepsilon _n^2}}. 
    \end{array}\]

\textbf{Verifying condition
  $\log N\left( {{\varepsilon _n},{\bB_n},\| \cdot \|_{L_2(Q_n)}} \right)
  \le n\varepsilon _n^2$}

Clearly,
$N\left( {{\bB_{n,{0^{d_n} }}},{\varepsilon _n},\|\cdot\|_{L_2(Q_n)}} \right)
= 2{M_n}/{\varepsilon _n}$.

In light of Lemma \ref{lemma:metric-entropy},
for ${M_n}\sqrt{{r_n}} >2 \varepsilon_n$, 
the metric entropy of 
${M_n}
\sqrt{{r_n}} \cH_1^{r_n,\gamma } + {\varepsilon _n}{\bB_{1,\gamma }}$ is bounded above: 
\[\begin{array}{lll}
   && \log N\left( {2{\varepsilon _n},
    {M_n}\sqrt{{r_n}} \cH_1^{r_n,\gamma }
    + {\varepsilon _n}{\bB_{1,\gamma }}
    ,\| \cdot \|_{L_2(Q_n)}} \right) \\ 
    &\le&
          \log N\left( {{\varepsilon _n},
          {M_n}\sqrt{{r_n}} \cH_1^{r_n,\gamma },
          \| \cdot \|_{L_2(Q_n)}} \right) 
    \\
    &\precsim &
                r_n^{|\gamma |}
                \log {\left( {{M_n}
                \sqrt{ r_n} {\varepsilon _n^{-1}}} \right)^{|\gamma | + 1}}
                / |\gamma|!. 
  \end{array}\]
Since $\log \left( {{M_n}\sqrt{r_n} {\varepsilon _n^{-1}}} \right)
\asymp  \log \left( {1/{\varepsilon _n}} \right) \asymp  \log (n)$,
the above metric entropy is further bounded above by
\[ \log N\left( {2{\varepsilon _n},
    {M_n}\sqrt{r_n } \cH_1^{r_n,\gamma }
    + {\varepsilon _n}{\bB_{1,\gamma }},
      \|\cdot\|_{L_2(Q_n)}} \right)
  \precsim
  r_n^{|\gamma |}\log {\left( {1/{\varepsilon _n}} \right)^{|\gamma | + 1}}
  \asymp  n\varepsilon _n^2. \]
The bound holds for all $\gamma$ due to the choice of $r_n$. 
 

Then note the bound  
$\log \left( {\sum\nolimits_{i = 1}^n {{x_i}} } \right)
\le
\log n +
{\log \left( {\mathop {\max }\nolimits_i {x_i}} \right)} $,  
and the bound $\log(x\vee y) \le \log(x) + \log(y)$ for $x,y>1$,  
it follows,
\[\begin{array}{lll}
    \log N\left( {{\varepsilon _n},{\bB_n},\| \cdot \|_2} \right)
    &\precsim & 
               \log ({2^{\underline{d}_n}}-1) + 
                \frac{1}{2} n\varepsilon _n^2 +
               \log (\frac{{2M_n}}{\varepsilon_n }) 
  \end{array}\]
where $\underline{d}_n \prec n\varepsilon _n^2 $ by construction and
$\varepsilon _n$ is some multiple of
${n^{ - 1/\left( {2 + {d_0}/\beta } \right)}}{\log ^\kappa }\left( n \right)$. 

\end{proof}
 
\subsection{Proof of Proposition \ref{proposition:FP-prior-prob}}
\label{sec:FP-prior-prob-proof}
 
Before the proof of Proposition \ref{proposition:FP-prior-prob}, 
first review some basic properties of Lambert $W$ function.
Lambert $W$ function defines the product log function:
\[x{e^x} = y \Leftrightarrow x = W\left( y \right). \]
In the proof of Proposition \ref{proposition:FP-prior-prob},
we need to solve an equation like 
$\left( {x - {c_1}} \right){e^x} = {c_2}$ for $x$.
With Lambert $W$ function, the solution is in closed form:
${x^*} = W\left( {{c_2}{e^{ - {c_1}}}} \right) + {c_1}$. 
The following lemma is useful for our calculations;
a proof can be found in the end of the Supplement. 
  \begin{supplemma}{5}
    \label{lemma:LambertW} 
The Lambert $W(\cdot)$ function is strictly increasing. If $x = W\left( y \right)>1$, then \(y > {e^{W(y)}} > y/\log \left( y \right)\). Furthermore, if $x = W(y) \in (0,1]$, then \(y \ge W\left( y \right) \ge y/e.\)
\end{supplemma}

With the above preparation,
here comes the proof of Proposition \ref{proposition:FP-prior-prob}. 

\paragraph{ Proof of Proposition \ref{proposition:FP-prior-prob}}

\begin{proof}
  As in Lemma 1 of \cite{castillo2008lower},
  Lemma 1 in \cite{ghosal2007convergence},
  or Lemma 5 in \cite{barron1999consistency}, 
  it suffices to show
  for some ${\varepsilon _n}$ with 
  ${\varepsilon _n}\to 0$ and $n \varepsilon _n^2\to \infty $, the following holds
  \begin{equation}
    \label{eq:lemma-FP-fraction}
    \frac{
      \Pi_n \left( {\Gamma  \in \fp (\gamma_n^*),{E_n}} \right)}
    {\Pi_n \left( {{B_{n}}\left( {{f_n^*},{\varepsilon _n}} \right)} \right)} 
    =
    o \left( {e^{ - 4n\varepsilon _n^2}} \right)
  \end{equation}
  where
  ${E_n}
  = \left\{ {f: \|f - {f_n^*}\|_{L_2(Q_n)} \le M{\varepsilon _n}} \right\}$
  as defined in the main text, 
  and 
  ${B_{n}}\left( {{f_n^*},{\varepsilon _n}} \right)
  = \left\{ {f: 
      KL \left( {{P_{{f_n^*}}},{P_f}} \right) \le \varepsilon _n^2,
        {V_{2,0}}\left( {{P_{{f_n^*}}},{P_f}} \right) \le \varepsilon _n^2} \right\}$
    is the KL neighborhood of the truth. 
    To prove (\ref{eq:lemma-FP-fraction}), we show the following:
    \[\begin{array}{l} 
        \sum\nolimits_{\gamma  \in \fp (\gamma_n^*),
        {|\gamma | \precsim \varepsilon _n^{ - 2/\alpha }}}
        \Pi_n \left( {\|W^{A,\gamma} - {f_n^*}\|_{L_2(Q_n)}
        \le M{\varepsilon _n}} \right)
        \le {e^{ - 5n\varepsilon _n^2}}
        \\
        \sum\nolimits_{\gamma  \in \fp (\gamma_n^*),
        |\gamma | \succsim \varepsilon _n^{ - 2/\alpha }}
        \Pi_n  \left( {|\Gamma | = |\gamma |} \right)
        \le {e^{ - 5n\varepsilon _n^2}}
        \\
        \Pi_n \left( {{B_{n}}\left( {{f_n^*},{\varepsilon _n}} \right)} \right)
        \ge
        \Pi_n \left( {{B_{n}}\left( {{f_n^*},{\varepsilon _n}} \right)}
        \cap \{\Gamma = \gamma_n^*\} \right)
        \ge {e^{ - n\varepsilon _n^2}}
      \end{array}\]

    \textbf{Denominator}  
        
    In the nonparametric regression with Normal error,
    KL divergence is equivalent to $L_2(Q_n)$ norm: 
    $K\left( {{P_{{f_n^*}}},{P_f}} \right)
    = \frac{1}{{2{\sigma ^2}}}\|{f_n^*} - f\|_{L_2(Q_n)}^2$ and
    the KL variation 
    ${V_{2,0}}\left( {{P_{{f_n^*}}},{P_f}} \right)
    = \frac{1}{{{\sigma ^2}}}\|{f_n^*} - f\|_{L_2(Q_n)}^2$. 
    Therefore, the denominator has the lower bound in $L_2(Q_n)$ norm: 
    \[\Pi_n \left( {{B_n}\left( {{f_n^*},{\varepsilon _n}} \right)} \right)
      \ge
      \Pi_n \left(\|W^{A,\gamma_n^*} -
          {f_n^*}\|_{L_2(Q_n)} \le \sigma {\varepsilon _n} \right)
      \Pi_n \left( {\Gamma  = {\gamma_n^*}} \right) \]
    where     $\Pi_n \left( {\Gamma  = {\gamma_n^*}} \right) $
    is the prior probability of the true model. 
    In light of the concentration inequality,
    Lemma \ref{lemma:centered-sbp} and  
    Lemma \ref{lemma:decentering-upper-bd},
    \[\begin{array}{lll}
       && \Pi _n \left( 
        \|W^{A,\gamma_n^*} - {f_n^*}\|_{L_2(Q_n)} \le \sigma
          {\varepsilon _n} \right)\\ 
        &\ge& \int_{K_n}
              ^{2K_n}
              {\Pi_n \left( 
              \|W^{a,\gamma_n^*}  - {f_n^*}\|_{L_2(Q_n)} \le \sigma {\varepsilon _n}| a \right)
              \Pi_n\left( da \right)} \\
        &\ge&
              \int_{K_n}
              ^{2K_n }
              {{e^{ - \phi _{{f_n^*}}^{a,{\gamma_n^*}}\left( {\sigma {\varepsilon _n}/2} \right)}}
              \Pi_n \left( da \right)} \\
        &\ge&
              \int_{K_n }
              ^{2K_n }
              {{e^{
              - \left( {C{a^{d_0}}
              + {C^\prime}{a^{d_0}}
              \log {{\left( {a/{\varepsilon _n}} \right)}^{d_0 + 1}}} \right)}}
              \pi_n\left( a \right)da} \\
        &\ge&
              {e^{ - C{\varepsilon_n ^{ - d_0/\beta }}
              \log {{\left( {1/{\varepsilon _n}} \right)}^{d_0 + 1}}}}  
      \end{array}\]   
    where $K_n = C{{\left( {1/{\varepsilon _n}} \right)}^{1/\beta }}$, 
    ${\varepsilon _n} \asymp  {n^{ - 1/\left( {2 + {d_0}/\beta } \right)}}
    (\log n) ^{\kappa_1} $ and 
    ${\kappa _1} =\frac{|\gamma_n^*|+1}{2 + {d_0}/\beta } $. 
    Note for $\gamma_n^*=0^{d_n}$, the prior on the mean is
    $W \sim N\left( {0,1} \right)$. 
    Denote standard Normal density function as $\phi(\cdot)$, then
    \[\begin{array}{lll} 
        \Pi_n \left(
        \|W - {f_n^*}\|_{L_2(Q_n)} \le \sigma {\varepsilon _n} \right)
        &=&
            \Pi_n \left( {{f_n^*} - \sigma {\varepsilon _n}
            \le f
            \le {f_n^*} + \sigma {\varepsilon _n}} \right)\\
        &\ge& 2\sigma {\varepsilon _n}
              \phi \left( {\left( {{f_n^*} + \sigma {\varepsilon _n}} \right)
              \vee \left( {{f_n^*} - \sigma {\varepsilon _n}} \right)} \right)\\
        &\ge& {e^{ - n\varepsilon _n^2}}
      \end{array}\]
    holds for every sufficiently large $n$.

    \textbf{Numerator} 

    In light of the concentration inequality
    (\ref{eq:concentration-inequality}),
    \begin{eqnarray}
      \label{eq:FP-Num}
      & &
          \Pi_n \left(
          {\|W^{A,\gamma} - {f_n^*}\|_{L_2(Q_n)}
          \le M{\varepsilon _n}} \right)
          \nonumber \\  
      &=& \int 
          {\Pi_n \left( {\|W^{A,\gamma} - {f_n^*}\|_{L_2(Q_n)} \le M{\varepsilon _n}
          | a} \right)
          \Pi_n \left( {da} | \gamma \right)} \nonumber \\ 
      &\le& \int_{0}^\infty
            {{e^{ - \phi _{{f_n^*}}^{a,\gamma }\left( {M{\varepsilon _n}} \right)}}
            \Pi_n \left( {da}|\gamma \right)},        
    \end{eqnarray} 
    where $\Pi_n(\cdot | \gamma)$ denotes the prior on $A$ conditional on model index parameter $\gamma$. 
    The sufficiently large constant $M$ can be absorbed into $\varepsilon_n$ by choosing it to be a large multiple of 
    $n^{-1/(2+d_0/\beta)}(\log n) ^{\kappa_1}$. For the ease of notation, henceforth $M$ is absorbed into $\varepsilon_n$.
    
    The integral (\ref{eq:FP-Num}) is further bounded by three parts:
    \begin{equation}
      \label{eq:lemma-FP}
      \int_0^{1/\xi } {
        {e^{ - \phi _{{f_n^*}}^{a,\gamma }\left( {{\varepsilon _n}} \right)}}
        \Pi_n \left( {da} |\gamma \right)}
      + \int_{1/\xi }^{\tau_n }
      {{e^{ - \phi _{{f_n^*}}^{a,\gamma }\left( {{\varepsilon _n}} \right)}}
        \Pi_n \left( {da} |\gamma\right)}
      + \Pi_n \left( {A >\tau_n } |\gamma\right)
    \end{equation}
    where $ \tau_n = C{\varepsilon _n}^{ - 1/\beta }$ for some
    constant $C$, 
    $\xi$ is a parameter of the dominating measure $Q_n$. 
     The three quantities are bounded respectively as follows.

     The  first quantity of (\ref{eq:lemma-FP}) is 0 due to the prior
     on $A$.
     
     The third quantity of (\ref{eq:lemma-FP}) has the desired upper
     bound $e^{-4n\varepsilon_n^2}$, because 
     $\tau_n^{|\gamma| } \succ n \varepsilon_n^2$
     for every $\gamma \in FP(\gamma_n^*)$. 


  For the second quantity of (\ref{eq:lemma-FP}),
  by Lemma \ref{lemma:decentering} and Lemma \ref{lemma:centered-sbp},
  the integral is bounded above by 
  \begin{equation}
    \label{eq:FP-integral-2} 
    D_2 \tau_n^{|\gamma|}
    \mathop {\sup }
    \limits_{a: 1/\xi  \le a  \le {\tau _n}}
    {\exp{ \left\{ - C{(c_\xi a)^{|\gamma |}}
        {\varepsilon _n^2
          {e^{c\varepsilon _n^{ - 2/\alpha }/a^2 }}- {C_2}{a^{|\gamma |}}
          +C_3 \log(|\gamma|)} \right\} }}
     \end{equation}
     where the prior density of $A | \Gamma = \gamma$ by assumption
     is upper bounded by
     ${D_2}\tau_n ^{|\gamma|-1}{e^{ - {C_2}{a^{|\gamma |}}+C_3 \log(|\gamma|)}}$. 
     
The solution to the supremum problem is 
${\tilde a}=
  \sqrt {\frac{{{c_n}}}{W_n}}$ 
where ${c_n} = c\varepsilon _n^{ - 2/\alpha }$, 
${C_n} = C_2C^{-1}c_\xi^{-|\gamma|} \varepsilon _n^{ - 2}$, 
$W_n = {{W\left( {{C_n}{e^{ - |\gamma |/2}}|\gamma |/2} \right) + |\gamma |/2}}$,
and $W$ is Lambert $W$ function.
This is solved by solving first order condition 
and verifying second order condition. 
Then the quantity (\ref{eq:FP-integral-2}) is proportional to 
\begin{equation}
  \label{eq:FP-sup}
  \tau_n^{|\gamma|}
  \exp \left\{ {{ - C\varepsilon _n^2{({c_\xi\tilde a})^{|\gamma |}}
      \left( {
          e^{W_n} 
          + {C_n}} \right)
    +C_3 \log(|\gamma|)} } \right\}.
\end{equation}

In light of Lemma \ref{lemma:LambertW}, there are two cases for $W(\cdot)$ and $W_n$.
To apply Lemma \ref{lemma:LambertW},
we need to solve ${C_n}{e^{ - |\gamma |/2}}|\gamma |/2 = e$ for
$|\gamma|$.
When $c_\xi ^2e \le 1$, 
${C_n}{e^{ - |\gamma |/2}}|\gamma |/2 \ge e$ holds
for all $|\gamma|>d_0$.  
When $c_\xi ^2e > 1$, 
there exists only one solution larger than $d_0$
for all sufficiently large $n>N(d_0,\xi)$. 
The solution is upper bounded by
$ {\bar{\zeta} _n} \equiv
  \frac{4}{{1 + 2\log {c_\xi }}}
  (1+\Delta)
  \log \left( {C_2C^{-1}\varepsilon _n^{ - 1}} \right) $
for any constant $\Delta>0$, 
and is lower bounded by 
$ {\underline{\zeta} _n} \equiv
  \frac{4}{{1 + 2\log {c_\xi }}}
  \log \left( {C_2C^{-1}\varepsilon _n^{ - 1}} \right)$ 
for all sufficiently large $n>N(\Delta) \vee N(d_0,\xi)$.

The followup analysis is divided into three regimes. 
Regime I deals with the fixed dimension and slowly growing dimension case. 
Regime II and regime III take care of growing dimension cases with
different growth rates.

\textbf{Regime I: 
$W\left( {{C_n}{e^{ - |\gamma |/2}}|\gamma |/2} \right) \ge 1$. }

In this regime, ${{C_n}{e^{ - |\gamma |/2}}|\gamma |/2} \ge e$ and  
$d_0<|\gamma|<\bar{\zeta}_n \asymp \log (\varepsilon_n^{-1})$.
By Lemma \ref{lemma:LambertW}, 
\[  \frac{1}{2}|\gamma|  C_n 
  >
  {e^{W_n}}
  >
  \frac{1}{2}|\gamma| {C_n}
  /
  \left(\log \left(|\gamma| {C_n}/2 \right) \right), \]  
and
$ \log(C_n) + \log(|\gamma|) 
  \succsim
  W_n
  \succsim
   \log({C_n}) + \log(|\gamma|)  
  - \log (\log \left( C_n \right) + \log(|\gamma|/2) )$. 
Then 
the supremum (\ref{eq:FP-sup}) is  upper bounded by
\[  {\exp \left\{ { - {C_2}
         {{\left( {\frac{{\varepsilon _n^{ - 2/\alpha }}}
                 {{\log {C_n}
                     + \log \left( {|\gamma |} \right)}}} \right)}^{|\gamma |/2}}
         +|\gamma| \log(\tau_n) + C_3 \log(|\gamma|)} \right\}
     }. \] 
Note  $\bar{\zeta}_n > |\gamma | \ge d_0 + 1$, 
the integral (\ref{eq:FP-integral-2}) is further bounded above by
\[{D_2}
  {\exp \left\{ { - {C_2^\prime}{{\left(
            {{{\varepsilon _n^{ - 2/\alpha }}}{
                \log^{-1} (\varepsilon _n^{ - 1})}}
          \right)}
        ^{\left( {{d_0} + 1} \right)/2}}
  +  \bar{\zeta}_n \log(\tau_n) + C_3\log (\bar{\zeta}_n)} \right\}}\]
for some constant $C_2^\prime$ and all sufficiently large $n$.

Since  $\alpha  < \beta \left( {1 + 1/d_0} \right)$ and
${{\bar \zeta }_n} \asymp  \log {\tau _n} \asymp  \log \left( {\varepsilon _n^{ - 1}} \right)$, 
the following holds for all sufficiently large $n$ 
\[\begin{array}{lll}
   && \sum\nolimits_{\gamma  \in \fp (\gamma_n^*):
    W\left( {{C_n}{e^{ - |\gamma |/2}}|\gamma |/2} \right) \ge 1}
      {\Pi_n \left( 
      \|W^{A,\gamma} - {f_n^*}\|_{L_2(Q_n)} \le M{\varepsilon _n} \right)} \\ 
    &\le&
          {2^{{{\bar \zeta }_n} - {d_0}}}
          {D_2}
          {e^{ - C_2^{\prime \prime} {{\left( {{{\varepsilon _n^{ - 2/\alpha }}}
          {{\log^{-1} \left( {\varepsilon _n^{ - 1}} \right)}}} \right)}^{\left( {{d_0} + 1} \right)/2}}}}\\
    &\le &\frac{1}{2}{e^{ - 5n\varepsilon _n^2}}. 
\end{array}\]

\textbf{Regime II: 
  $W\left( {{C_n}{e^{ - |\gamma |/2}}|\gamma |/2} \right) < 1$
  and
  $|\gamma| \precsim \varepsilon _n^{ - 2/\alpha }$ 
}

In this regime, 
${C_n} < {e^{|\gamma |/2 + 1}}|\gamma |/2$,
$ |\gamma|> \underline{\zeta}_n$, and 
$|\gamma |/2 \le {W_n} \le |\gamma |/2 + 1$.
Then plugging in bounds on $W_n$, 
the supremum in (\ref{eq:FP-sup}) is  upper bounded by
${\exp \{ { - {C_2}
      {{\left( {\frac{{\varepsilon _n^{ - 2/\alpha }}}
              {{|\gamma |/2 + 1}}} \right)}^{ |\gamma |/2}}
    +|\gamma| \log(\tau_n) + C_3 \log(|\gamma|)} \} }$.  
Note  $  \varepsilon _n^{ - 2/\alpha } \succsim 
|\gamma |
>
\underline{\zeta}_n > d_0 $, 
the integral (\ref{eq:FP-integral-2}) is further bounded above by
\[\exp {\left\{ - C_2^\prime 
      {{\left( {\varepsilon _n^{ - 2/\alpha }/
              \left( {\underline{\zeta }_n/2} \right)} \right)}
        ^{\underline{\zeta} _n/2}}
      +\varepsilon _n^{ - 2/\alpha } \log(\tau_n)
      + C_3^\prime \log(\varepsilon _n^{ - 1 })
  \right\}}\]
for some constant $C_2^\prime$ and all sufficiently large $n$. 
   
Since  $\alpha  < \beta \left( {1 + 1/d_0} \right)$,
$|\gamma| \precsim \varepsilon_n^{-2/\alpha} \prec n \varepsilon_n^2$, and
${{\underline{ \zeta} }_n} \asymp  \log {\tau _n} \asymp  \log \left( {\varepsilon _n^{ - 1}} \right)$, 
the following holds for all sufficiently large $n$ 
\[\begin{array}{lll}
   && \sum\nolimits_{\gamma  \in \fp (\gamma_n^*): 
    W\left( {{C_n}{e^{ - |\gamma |/2}}|\gamma |/2} \right) < 1}
      {\Pi_n \left( {
      \|W^{A,\gamma} - {f_n^*}\|_{L_2(Q_n)} \le M{\varepsilon _n}} \right)} \\ 
    &\le&
          {2^{ { C\varepsilon_n^{-2/\alpha} -{\underline{ \zeta} }_n}}}
          {D_2}
          {e^{ - C_2^\prime {{\left( {{{\varepsilon _n^{ - 2/\alpha }}}
          /(\underline{\zeta}_n/2)
          } \right)}^{ \underline{\zeta}_n/2}}}}\\
    &\le &\frac{1}{2}{e^{ - 5n\varepsilon _n^2}}. 
\end{array}\]

\textbf{Regime III: 
  $W\left( {{C_n}{e^{ - |\gamma |/2}}|\gamma |/2} \right) < 1$
  and
  $|\gamma| \succsim \varepsilon _n^{ - 2/\alpha }$ 
}

In this regime,
 $\Pi_n \left( {|\Gamma | = |\gamma |} \right) \le {e^{ - C|\gamma {|^\rho }}}$
for some $\rho  \ge ({d_0}+1)/2$,
the following holds for some constant $C^\prime$ and all sufficiently large $n$ 
\[\Pi_n \left( { \Gamma \in \fp (\gamma_n^*): |\Gamma | \ge \varepsilon _n^{ - 2/\alpha }} \right)
  \le \Pi_n \left( {|\Gamma | \ge \varepsilon _n^{ - 2/\alpha }} \right) 
  \le {e^{ - {C^\prime}
      \varepsilon _n^{ - 2\rho /\alpha }}}
  \le {e^{ - 5n\varepsilon _n^2}}. \]

Therefore, combining the three regimes and the rest ingredients,
\[\Pi _n \left( \Gamma \in \fp (\gamma_n^*),
      \{   
      \|W^{A,\Gamma} - {f_n^*}\|_{L_2(Q_n)} \le M {\varepsilon _n}\}
         \right) 
  \le
  2 {e^{ - 5n\varepsilon _n^2}} \]
and the ratio (\ref{eq:lemma-FP-fraction}) holds. 
   
\end{proof}

\subsection{Proof of Lemma   \ref{lemma:metric-entropy}}
\begin{proof}
  With some abuse of notation, 
  denote (\ref{eq:RKHS}) by $\cH^{a,\gamma}$.
  By isometry, it suffices to compute the metric entropy of
  (\ref{eq:RKHS}) 
  with respect to the $\ell_2$ metric $||\cdot||_2$.
  
  The strategy of the proof is to construct a
  $\tilde{\cH}\subset \cH^{a,\gamma}$
  and use the metric entropy of $\tilde{\cH}$ as bounds 
  for the metric entropy of $\cH^{a,\gamma}$.
  For the ease of notation, superscript `` $(\gamma)$'' of
  the eigenvector and eigenvalue notation is dropped
  as the dependence on $\gamma$ is clear. 
  
  The eigenvalue $\mu_j$ takes the form  $(2v_1/V)^{|\gamma|/2} B^m$ for some $m \in \N$. 
  Solve \[ (2v_1/V)^{|\gamma|/2} B^m = \varepsilon^2 \] for $m \in \R $ and
  the solution denoted by $m^*$ is
  \[{m^*} = {\left( {\log \left( {1/B} \right)} \right)^{ - 1}}
  \left( {2\log \left( {1/\varepsilon } \right)
      - \frac{|\gamma|}{2}\log \left( {\frac{V}{{2{v_1}}}} \right)} \right).\] 
  The assumption
  $1/\varepsilon  \ge C_H{\left( {\xi a} \right)^{|\gamma|/2}}$
  guarantees
  ${m^*}
  \asymp
  {\left( {\log \left( {1/B} \right)} \right)^{ - 1}}
  \log \left( {1/\varepsilon } \right)$.  
  
  Define 
  $ \tau
  =  \sum\nolimits_{j = 0}^{\left\lfloor {{m^*}} \right\rfloor } \binom{|\gamma|+j-1}{|\gamma|-1}
  = \binom{\left\lfloor {{m^*}} \right\rfloor +|\gamma|}{|\gamma|}$  
  where $\left\lfloor {{m^*}} \right\rfloor$ is the greatest integer less than $m^*$. 
  By construction, $\sqrt{\mu_j} \ge \varepsilon$ for all $j\le \tau$ and
   $\sqrt{\mu_j} \le \varepsilon$ for all $j> \tau$. 

  Define $ {{\tilde \cH}_\varepsilon }
  = \left\{ {\theta  \in {\cH^{a,\gamma }}:{\theta _j} = 0,
      \forall j > \tau } \right\}$.  
  ${\tilde \cH}_\varepsilon$ contains the elements in $\cH^{a,\gamma}$
  whose components after $\tau$ vanish to 0. 
  Any $\varepsilon-$cover of $\tilde{\cH}_\varepsilon$ forms
  a $\sqrt{2} \varepsilon-$cover of $\cH^{a,\gamma}$.
  To see this, suppose $\{\theta^k \}_1^N$ forms a $\varepsilon-$cover of  $\tilde{\cH}$,
  then for any $\theta \in \cH^{a,\gamma }$, 
  \[\begin{array}{lll}
      \mathop {\min }\limits_k ||\theta  - {\theta ^k}||_2^2
      &=&
          \mathop {\min }\limits_k
          \sum\nolimits_{j = 1}^\tau
          {{{\left( {{\theta _j} - \theta _j^k} \right)}^2}}
          + \sum\nolimits_{j = \tau  + 1}^\infty  {\theta _j^2} \\
      &\le&
            {\varepsilon ^2}
            + {\mu _\tau }\sum\nolimits_{j = \tau  + 1}^\infty  {\theta _j^2/{\mu _j}} \\
      &\le& 2{\varepsilon ^2}. 
    \end{array}\]
  As the $\sqrt{2}$ scaling does not matter in metric entropy calculation,
  it suffices to work with $\tilde{\cH}$.  

  \textbf{\textup{Upper bound}}

  By construction,  
  $ {{\tilde \cH}_\varepsilon }
    \supseteq
    B_2^{\tau+1} \left( \varepsilon  \right)$ 
  where $B^k_q(\varepsilon)$ denotes $k-$dimensional $\ell_q-$ball of radius $\varepsilon$:
  $B_q^k\left( \varepsilon  \right)
  = \left\{ {x:\sum\nolimits_{i = 1}^k {x_i^q}  \le {\varepsilon ^q};
      {x_j} = 0, \forall j > k} \right\}$.     
  $\tau+1$ is due to eigenvalues start from $\mu_0$,

  Note 
  ${{\tilde \cH}_\varepsilon } + B_2^{\tau+1} \left( \varepsilon  \right)
  \subseteq 2{{\tilde \cH}_\varepsilon }$, 
  standard volume argument yields 
  \[{N_{[]}}\left( {\varepsilon ,{{\tilde \cH}_\varepsilon },|| \cdot |{|_2}} \right)
    vol\left( {B_2^{\tau+1} \left( {\varepsilon /2} \right)} \right)
    \le vol\left( {2{{\tilde \cH}_\varepsilon }} \right)\]
  where
  $ vol\left( {2{{\tilde \cH}_\varepsilon }} \right)
  = {2^{\tau+1} }vol\left( {{{\tilde \cH}_\varepsilon }} \right)
  \le {2^{\tau+1} }\prod\nolimits_{i = 0}^\tau  {\sqrt {{\mu _i}} }$ 
  and 
  $ vol\left( {B_2^{\tau+1} \left( {\varepsilon /2} \right)} \right)
  = {\left( {\varepsilon /2} \right)^{\tau+1} }
  vol\left( {B_2^{\tau+1} \left( 1 \right)} \right)$.  
  Then, it follows 
  \[\log {N_{[]}}\left( {\varepsilon ,{{\tilde \cH}_\varepsilon },|| \cdot |{|_2}} \right)
    \le
    2(\tau +1) \log 2
    + (\tau+1) \log \left( {1/\varepsilon } \right)
    + \frac{1}{2}\sum\nolimits_{i = 0}^\tau  {\log {\mu _i}} \] 
  where
  $\log (\mu_0) = \frac{{{|\gamma|}}}{2}\log \left( {2{v_1}/V} \right)$,
  and 
  $\log {\mu _i}
  = \log (\mu_0) + h \log B$
  for $h=1,..., \left\lfloor {{m^*}} \right\rfloor  $ and 
  $i \in
  [\binom{|\gamma|+h-1}{|\gamma|}, \binom{|\gamma|+h}{|\gamma|}) $. 
  The exponent $h$ is the multiplicity of eigenvalues. 

  
  Note $B<1$ and total multiplicity is  
  $ \sum\nolimits_{j = 0}^{\left\lfloor {{m^*}} \right\rfloor } j 
  \binom{|\gamma|-1+j}{|\gamma|-1}  
  \ge 
  \sum\nolimits_{j = 1}^{\left\lfloor {{m^*}} \right\rfloor } 
  \binom{|\gamma|-1+j}{|\gamma|-1} 
  = \tau -1 $,  it follows 
  \[\begin{array}{lll}
      \sum\nolimits_{i = 0}^\tau  {\log {\mu _i}}  
      &=& 
          \frac{{{|\gamma|}}}{2}\log \left( {2{v_1}/V} \right)\left( {\tau  + 1} \right)
          - \sum\nolimits_{j = 0}^{\left\lfloor {{m^*}} \right\rfloor } j 
          \binom{|\gamma|-1+j}{|\gamma|-1}
          \log \left( {1/B} \right)\\
      &\le&
            \frac{{{|\gamma|}}}{2}\log \left( {2{v_1}/V} \right)
            \left( {\tau  + 1} \right) 
            -
            \log \left( {1/B} \right)
            \left( {\tau - 1} \right) . 
    \end{array}\] 

 Note ${a \xi } \log \left( {1/\varepsilon } \right) > |\gamma|$ implies 
 $m^* \asymp m^*+|\gamma|$,
 the with the definition of $\tau$ and 
  $1/\varepsilon  
  \ge 
  C_H{\left( {\xi a} \right)^{|\gamma|/2}}$, 
  \[\begin{array}{lll}
      &&
         \log {N_{[]}}\left( {\varepsilon ,{{\tilde \cH}_\varepsilon
         }, || \cdot |{|_2}} \right)\\
      &\precsim&
                 (\tau+1) \left[ {\log \left( {1/\varepsilon } \right)
                 -
                 \frac{{{|\gamma|}}}{4}\log \left( \frac{V}{2v_1}\right)} \right]
                 - \frac{1}{2}(\tau-1)
                 \log \left( {1/B} \right) \\
      &\precsim&
                 {\left( {{m^*} + {|\gamma|}} \right)^{{|\gamma|}}}
                 \log \left( {1/\varepsilon } \right)/{|\gamma|}! \\
      &\precsim&
                 {\left( \log(1/B) \right)^{{-|\gamma|}}}
                 \log {\left( {1/\varepsilon } \right)^{{|\gamma|} + 1}}/{|\gamma|}!.
    \end{array}\]

  \textbf{\textup{Lower bound}}

  Since $\tilde{\cH}_\varepsilon \subset \cH^{a,\gamma}$,
  lower bound of $\tilde{\cH}_\varepsilon $'s metric entropy also lower bounds 
  $ \cH^{a,\gamma}$'s metric entropy. 
  The fact ${\tilde{\cH}_\varepsilon } \supseteq B_2^{\tau  + 1}\left( \varepsilon  \right)$
  implies 
  \[\begin{array}{lll}
      \log N\left( {{\tilde{\cH}_\varepsilon },\varepsilon /2,|| \cdot |{|_2}} \right)
      &\ge&
            \log N\left( {B_2^{\tau  + 1}\left( \varepsilon  \right),\varepsilon /2,|| \cdot |{|_2}} \right)\\
      &\ge&
            \left( {\tau  + 1} \right)\log 2\\
      &\asymp&
               {\left( \log(1/B) \right)^{-{|\gamma|}}}
               \log {\left( {1/\varepsilon } \right)^{{|\gamma|}}}/|\gamma|,  
    \end{array}\]
  and
  $\log N\left( {\cH^{a,\gamma},\varepsilon ,|| \cdot |{|_2}} \right)
      \succsim
      {\left( {\log(1/B)  } \right)^{-|\gamma|}}
      \log {\left( {1/\varepsilon } \right)^{|\gamma|}}/|\gamma|!$.

\textbf{\textup{Bounding $\log(1/B)$}}

    Note 
  \[\begin{array}{lll}
      \log \left( {1/B} \right) 
      &=& \log \left( {1 + {v_1}/{v_2} + {v_3}/{v_2}} \right)\\
      &=& \log \left( {1 + {v_1}/{v_2} 
          + \sqrt {{{\left( {{v_1}/{v_2}} \right)}^2} 
          + 2{v_1}/{v_2}} } \right)
    \end{array}\]
  where ${v_1}/{v_2} = 1/\left( {4{a^2}{\xi ^2}} \right)$.

  When $a$ is close to 0, $\log (1/B) \asymp \log (v_1/v_2)
  \asymp \log (1/a) \precsim 1/a$.
  
  When $a$ is sufficiently large, ${v_1}/{v_2}$ is close to 0 
  and  $\sqrt {{{\left( {{v_1}/{v_2}} \right)}^2} + 2{v_1}/{v_2}} \succ {v_1}/{v_2}$. 
  With the relation $\log(1+x) \approx x$ for $x \approx 0$,
  it follows $\log {\left( {1/B} \right)^{ - 1}} \asymp  {a}{\xi}$.

  For more precise characterization,
  there exists constants $C_1, C_2$ and $C_3$ such that 
  \[{C_1}a \le \log {\left( {1/B} \right)^{ - 1}} \le {C_2} + {C_3}a\]
  holds for all $a>0$. The constants only depend on $\xi$. 
  
\end{proof}

\subsection{Proof of Lemma \ref{lemma:decentering-upper-bd}}

\begin{proof}

  The proof extends Lemma 7 of \cite{vaart2011information}
  and allows for rescaling and different dimensions. 
  Following the representation theorem,
  for $\lambda_0 \in \R^{d_0}$, 
  let
  \[ \psi \left( \lambda_0  \right)
  = \frac{{{{\hat f}_0}\left( {{\lambda_0  }} \right)}}
  {{{m_{a,\gamma_n^* }}\left( {{\lambda_0}} \right)}}
  {1_{\left( \lambda_0 \in \R^{d_0}:
        {||\lambda_0 |{|_2} < K} \right)}},\] 
  then for every $t \in \R^{d_n}$, 
  $ {{f}_n^*}\left( t \right) - {h_\psi }\left( t \right)
  = \int_{\lambda_0 \in \R^{d_0}:
    ||\lambda_0 |{|_2} \ge K}
    {{e^{i\left( {\lambda_0 ,t_0} \right)}}
      {{\hat f}_0}\left( {{\lambda_0 }} \right) 
      d\lambda_0 }$. 
    By \Holder's inequality, 
    \[\begin{array}{lll}
      |{h_\psi }\left( t \right) - {{f}_n^*}\left( t \right)|^2
      &\le &
             \int_{{\R^{d_0}}}^{}
             {|{{\hat f}_0}\left(\lambda_0  \right)| ^2
             1\left( {||\lambda_0 |{|_2} \ge K} \right)d\lambda_0 } \\
      &\le&
            ||{{f}_0}|{|^2_{{H^\beta }\left( {{\R^{d_0}}} \right)}}
            \mathop {\sup }\limits_{\lambda_0  \in {\R^{d_0}}}
            {\left( {1 + ||\lambda_0 ||_2^2} \right)^{ - \beta }}
            1\left( {||\lambda_0 |{|_2} \ge K} \right) \\
      &\precsim &
                  ||{{f}_0}|{|^2_{{H^\beta }\left( {{\R^{d_0}}} \right)}}
                  {K^{ - 2\beta}} . 
    \end{array}\]
  Therefore,  
  \[\begin{array}{lll} 
      ||{h_\psi } - {{f}_n^*}||_{L_2(Q_n)}^2
      &=& \int_{\R^{d_n}}
          {|{h_\psi }\left( t \right) - {f_n^*}\left( t \right){|^2}
          g_{d_n}\left( t \right)dt} \\ 
      &\precsim&
                 ||{{f}_0}||_{{H^\beta(\R^{d_0 }) }}^2
                 {K^{ - 2\beta}}. 
    \end{array}\]
   Choose ${K^{ - \beta  }} \propto {\varepsilon  }$ such that 
  $||{h_\psi } - {{f}_n^*}|{|_{L_2(Q_n)}} \le {\varepsilon  }$.  
  \[\begin{array}{lll}
      ||{h_\psi }||_{{\cH^{a,\gamma_n^* }}}^2
      &=& 
          \int_{\lambda_0  \in {\R^{d_0}}: ||\lambda_0 |{|_2} < K}^{} 
          {|{{\hat f}_0}\left( {{\lambda_0 }} \right){|^2}
          {m_{a,\gamma_n^* }}{{\left( {{\lambda_0 }} \right)}^{ - 1}}
          d\lambda_0 } \\ 
      &=&
          \int_{\lambda_0  \in {\R^{d_0}}:||\lambda_0 |{|_2} < K}^{}
          {|{{\hat f}_0}\left( \lambda_0  \right){|^2}
          {{\left( {2\sqrt \pi  } \right)}^{d_0}}
          {a^{d_0}}
          {e^{\frac{1}{{4{a^2}}}||\lambda_0 ||_2^2}}
          d\lambda_0  } \\
      &\le&  
            ||{{f}_0}||_{{H^\beta }\left( {{\R^{d_0}}} \right)}^2
            {\left( {2\sqrt \pi  } \right)^{d_0}}{a^{d_0}}
            {e^{\frac{1}{{4{a^2}}}{K^2}}} . 
    \end{array}\]
 Then it  holds that 
  \[|| {h_\psi }||_{{\cH^{a,\gamma_n^* }}}^2
    \precsim 
    {\left( {2\sqrt \pi  } \right)^{d_0}}
    {a^{d_0}} 
    {e^{C\varepsilon  ^{ - 2/\beta }/{a^2}}}\]
  
\end{proof}

\subsection{Proof of Lemma   \ref{lemma:decentering}}

\begin{proof}

  By Parseval's identity,
  for $h_\psi \in \cH^{a,\gamma}$,  
  \[ ||{h_\psi } -  {{f}_n^*}||_{L_2(Q_n)}^2
    = ||{\left( {{h_\psi } - f_n^*} \right)}\sqrt {g_{d_n}} ||_2^2
    = ||\widehat {h_\psi\sqrt {g_{d_n}} }
    - \widehat {f_n^*\sqrt{ g_{d_n}} }||_2^2.\]
  Therefore,
  $||{h_\psi } - f_n^*||_{L_2(Q_n)}< \varepsilon  $ implies 
  $||\widehat {{h_\psi }\sqrt {g_{d_n}} }{\chi _K}
    - \widehat {f_n^*\sqrt {g_{d_n}}}{\chi _K}||_2
    < \varepsilon $,  
  where ${\chi _K} = \left\{ {\lambda \in\R^{|\gamma|} :||\lambda |{|_2} > K} \right\}$. 
  Triangle inequality implies
  \begin{equation}
    \label{eq:lemma-decentering}
    ||\widehat {{h_\psi }\sqrt {g_{d_n}} }{\chi _K}|{|_2}
    > ||\widehat {f_n^*\sqrt {g_{d_n}} }{\chi _K}|{|_2} - {\varepsilon  }.
  \end{equation}
  
  For the right hand side of inequality (\ref{eq:lemma-decentering}),
  denote  $\R^{|\gamma|}\backslash \R^{d_0}$
  as $\R^{d_1}$, 
  with the assumption on ${f}_0$, 
  \[\begin{array}{lll}  
      && ||\widehat {f_n^*\sqrt {g_{d_n}} }{\chi _K}||_2^2 \\ 
      &\propto &
          \int \int_{(\lambda_0 , \lambda_1) \in\R^{|\gamma|}:
          ||\lambda_0 |{|_2^2} +||\lambda_1||_2^2 > K^2}
          {|\widehat {{{f}_0}\sqrt {g_{d_0}} }\left( \lambda_0  \right){|^2}}
          |\widehat{\sqrt{g_{d_1}}}(\lambda_1) |^2
          d\lambda_0 d \lambda_1 \\
      &\ge & 
             \int \int_{(\lambda_0 , \lambda_1) \in\R^{|\gamma|}: 
             ||\lambda_0 |{|_2} > K}
             {|\widehat {{{f}_0}\sqrt {g_{d_0}} }\left( \lambda_0  \right){|^2}}
             |\widehat{\sqrt{g_{d_1}}}(\lambda_1) |^2
             d\lambda_0 d \lambda_1 \\ 
      &\succsim&
                 \int_{\lambda\in\R^{d_0}: 
                 ||\lambda_0 |{|_2} > K}^{}
                 {||\lambda_0 |{|^{ -(2\alpha+d_0)}}
                 d\lambda_0 } \\
      &\succsim&
              {\left( {1/K} \right)^{2\alpha}} 
    \end{array}\]
  where $K\propto \varepsilon  ^{ - 1/\alpha}$ is chosen
  such that $||\widehat {f_n^*\sqrt {g_{d_n}} }{\chi _K}||_2$
  is lower bounded by $2\varepsilon $.

  By Lemma 16 of \cite{vaart2011information} and the representation of $h_\psi$,  
  \[\begin{array}{lll}
      && ||\widehat {{h_\psi }\sqrt {g_{d_n}} }{{\chi} _{2K}}|{|_2} \\ 
      &=&
          ||(\left( \psi {m_{a,\gamma }}  \right) * \widehat {\sqrt {g_{d_n}} })
          {  \chi _{2K}}|{|_2}\\
      &\le&
            ||\psi {m_{a,\gamma }}{ \chi _K}| {|_2}
            ||\widehat {\sqrt {g_{d_n}} }\left( {1 - { \chi _K}} \right)|{|_1}
            +
            ||\psi {m_{a,\gamma }}|{|_2}
            ||\widehat {\sqrt {g_{d_n}} }{ \chi _K}|{|_1}. 
    \end{array}\]
  The terms on the above right hand side are bounded in the following way, 
  \[\begin{array}{lll}
      ||\psi {m_{a,\gamma }} { \chi _K}||_2^2
      &=&
          \int_{\lambda \in \R^{|\gamma|}:
          ||\lambda |{|_2} > K}^{}
          {|\psi \left( \lambda  \right){|^2}
          {m_{a,\gamma }}{{\left( \lambda  \right)}^2}
          
          d\lambda } \\
      &\le&
            {m_{a,\gamma }}\left( K \right)
            \int_{\lambda \in \R^{|\gamma|}: ||\lambda |{|_2} > K}^{}
            {|\psi \left( \lambda  \right){|^2}
            {m_{a,\gamma }}\left( \lambda  \right)d\lambda } \\
      &\le&
            {\left( {2\sqrt \pi  } \right)^{ - |\gamma |}}
            {a^{ - |\gamma |}}{e^{ - \frac{1}{4}{K^2}/{a^2}}}
            ||{h_\psi }||_{{\cH^{a,\gamma }}}^2 
    \end{array}\]
  \[\begin{array}{lll}
      ||\widehat {\sqrt {g_{d_n}} }\left( {1 - { \chi _K}} \right)|{|_1}
      &=&
          \int_{\lambda \in \R^{|\gamma|}: ||\lambda |{|_2} \le K}^{}
          {{{\left( {2\sqrt {2\pi } \xi } \right)}^{|\gamma|/2}}
          {e^{ - {\xi ^2}||\lambda ||_2^2}}d\lambda } \\
      &<&
          \int_{\R^{|\gamma|}} {{{\left( {2\sqrt {2\pi } \xi } \right)}^{|\gamma|/2}}
          {e^{ - {\xi ^2}||\lambda ||_2^2}}d\lambda }\\
      & \precsim &
                   (2\sqrt{2\pi})^{|\gamma|/2} \xi ^{-|\gamma|/2} 
          < \infty 
    \end{array}\]
  \[\begin{array}{lll}
      ||\psi {m_{a,\gamma }  }||_2^2
      &=&
          \int_{{\R^{|\gamma|} }}^{} {|\psi \left( \lambda  \right){|^2}
          {m_{a,\gamma }}{{\left( \lambda  \right)}^2}d\lambda } \\
      &\le&
            {\left( {2\sqrt \pi  } \right)^{ - |\gamma |}}
            {a^{ - |\gamma |}}
            \int_{{\R^{|\gamma|} }}^{} {|\psi \left( \lambda  \right){|^2}
            {m_{a,\gamma }}\left( \lambda  \right)d\lambda } \\
      &=& {\left( {2\sqrt \pi  } \right)^{ - |\gamma |}}
          {a^{ - |\gamma |}}||{h_\psi }||_{{\cH^{a,\gamma }}}^2
    \end{array}\]

  \[\begin{array}{lll}
      ||\widehat {\sqrt {g_{d_n}} }{ \chi _K}|{|_1}
      &=&
          \int_{\lambda \in \R^{|\gamma|}:
          ||\lambda |{|_2} > K}
          {{{\left( {2\sqrt {2\pi } \xi } \right)}^{|\gamma|/2}}
          {e^{ - {\xi ^2}||\lambda ||_2^2}}d\lambda } \\
      &\le&
            {e^{ - {\xi ^2}{K^2}/8}}
            \int_{\lambda \in \R^{|\gamma|}: ||\lambda |{|_2} > K}
            {{{\left( {2\sqrt {2\pi } \xi } \right)}^{|\gamma|/2}}
            {e^{ - \frac{7}{8}{\xi ^2}||\lambda ||_2^2}}d\lambda } \\
      &\le&
            {e^{ - {\xi ^2}{K^2}/8}}
            \int {{{\left( {2\sqrt {2\pi } \xi } \right)}^{|\gamma|/2}}
            {e^{ - \frac{7}{8}{\xi ^2}||\lambda ||_2^2}}d\lambda } \\
      &\precsim&  (2\sqrt{2\pi})^{|\gamma|/2} \xi ^{-|\gamma|/2} {e^{ - {\xi ^2}{K^2}/8}}
    \end{array}\]

  Since a factor of 2 does not matter, combining bounds of RHS and LHS
  of inequality (\ref{eq:lemma-decentering})
  yields the following inequality  
  \[{a^{ - |\gamma |/2}}
    (\sqrt{2}/ \xi) ^{|\gamma|/2}
    \left( {{e^{ - \frac{1}{8}{K^2}/{a^2}}}
        + {e^{ - \frac{1}{8}{\xi ^2}{K^2}}}} \right)
    ||{h_\psi }||_{{\cH^{a,\gamma }}}
    \succsim
     ||\widehat {{h_\psi }\sqrt {g_{d_n}} }{\chi _{2K}}|{|_2}
    \ge {\varepsilon  }\]
  which is
  \[||{h_\psi }||_{{\cH^{a,\gamma }}}^2
    \succsim
    \varepsilon ^2
    {a^{|\gamma |}} c_\xi^{|\gamma|} 
    {e^{\frac{1}{4}{K^2}\left( {{a^{ - 2}} \wedge {\xi ^2}} \right)}} \]
  where $c_\xi= \xi / \sqrt{2}$. 
   
\end{proof}

\subsection{Proof of Lemma \ref{lemma:LambertW}}

\begin{proof}
  Let $f(x)=xe^x$ and $f$ is a strictly increasing, smooth map:
  $[0,\infty) \to [0,\infty)$. So its inverse is also strictly
  increasing. 

For the second claim, the first inequality follow because by the definition of $W(\cdot)$ and the assumption $x>1$,
  \[y=W(y)e^{W(y)} >e^{W(y)}. \]
  That is,
  $\log(y) >W(y)$.
  Then, $\log(y) e^{W(y)} >W(y) e^{W(y)} = y$, which
  is the second inequality.

Finally, as $W(y) \in (0,1]$, by monotonicity,
  $W\left( y \right)
  \le
  W\left( y \right){e^{W\left( y \right)}}
  \equiv
  y
  \le W\left( y \right)e$,
  that is, \[y \ge W\left( y \right) \ge y/e.\]
\end{proof}

\end{document}